\documentclass[aos]{imsart}

\RequirePackage{amsthm,amsmath,amsfonts,amssymb}
\RequirePackage[numbers]{natbib}
\RequirePackage[colorlinks,citecolor=blue,urlcolor=blue]{hyperref}
\RequirePackage{graphicx}
\usepackage{./aoslay}

\startlocaldefs
\numberwithin{equation}{section}
\theoremstyle{plain}

\newtheorem{problem}{Problem}

\newcommand{\eas}{e.a.s.}
\newcommand{\ellht}{\ell^{\mathsf{HT}}}
\newcommand{\ellip}{\ell^{\mathsf{IP}}}
\newcommand{\ellol}{\ell^{\mathsf{OL}}}

\endlocaldefs

\begin{document}

\begin{frontmatter}
\title{Prediction with eventual almost sure guarantees}
\runtitle{Prediction with eventual almost sure guarantees}

\begin{aug}
\author[A]{\fnms{Changlong} \snm{Wu}\ead[label=e1, mark]{wuchangl@hawaii.edu}}
\and
\author[A]{\fnms{Narayana} \snm{Santhanam}\ead[label=e2, mark]{nsanthan@hawaii.edu }}


\address[A]{Department of Electrical Engineering,
University of Hawai'i at Manoa,
\printead{e1,e2}}

\end{aug}

\begin{abstract}
We study the problem of sequentially predicting properties of a probabilistic model and its next outcome over an infinite horizon, with the goal of ensuring that the predictions incur only finitely many errors with probability 1. We introduce a general framework that models such prediction problems, provide general characterizations for the existence of successful prediction rules, and demonstrate the application of these characterizations through several concrete problem setups, including hypothesis testing, online learning, and risk domination. In particular, our characterizations allow us to recover the findings of Dembo and Peres (1994) with simple and elementary proofs, and offer a partial resolution to an open problem posed therein.
\end{abstract}


\end{frontmatter}

\section{Introduction}


Suppose we sequentially observe a sample \(X_1, X_2, \ldots\) from an unknown probabilistic model. At each time step \(n\), we make a prediction about either a property of the underlying model or the next sample \(X_{n+1}\), based on the current observation \(X_1, \ldots, X_n\) of the sample. There is a predefined binary \(\{0,1\}\)-loss that measures whether the prediction is correct/acceptable or not. The natural questions then are: under what conditions will we be able to incur only finitely many losses with probability 1, so that the predictions are eventually correct almost surely? If we can do so, can we identify from what point we will not incur any more losses?

This problem setup was initiated by Cover~\cite{cover1973determining}. In Cover's paper, the objective is to predict the irrationality of the mean of an underlying distribution over \([0,1]\) using $i.i.d.$ samples of it. The prediction can be updated after every observation, but the learner is allowed only finitely many errors with probability 1. It is shown in~\cite{cover1973determining} that such a prediction rule exists for distributions with means in a subset of \([0,1]\) that contains all rational numbers and has Lebesgue measure 1. Cover's setup was generalized substantially by Dembo and Peres~\cite{dembo1994topological} to identify general properties of distributions over \(\mathbb{R}^d\) using a topological criterion. Similar problem setups were continued in~\cite{kulkarni1994paradigm, koplowitz1995cover, leshem2006cover, naaman2016almost}. However, all of these setups considered only the classification of the underlying distribution using i.i.d. samples.

In~\cite{wu2019isit}, the authors considered a problem setup of predicting upper bounds on the \emph{next} observation of i.i.d. samples from distributions over \(\mathbb{N}\), such that the next observation violates the bound only finitely often with probability 1. In~\cite{JMLR:v16:santhanam15a}, the authors obtain a stopping rule that additionally indicates when such a prediction has made the last mistake.

This paper introduces a general framework that encompasses the previously mentioned problem setups within a unified analysis framework. Informally, we consider the following general prediction setup. The observations are modeled as a general discrete-time random process \(X_1, X_2, \ldots\) whose underlying probability measure (not necessarily i.i.d.) \(p\) belongs to a known collection \(\mathcal{P}\). The learning process is a game between Nature and a learner, where the learner attempts to predict a property of \(p\) or future observations. Nature fixes a random process \(p \in \mathcal{P}\) at the beginning of the game. At each time step \(n\), the learner makes a prediction \(Y_n\) using \(X_1, \ldots, X_{n-1}\). The prediction, the next realization \(X_n\), and potentially the underlying process \(p\) are associated with a binary 0-1 loss \(\ell\), where 1 indicates an error/unsatisfactory prediction. The collection \(\mathcal{P}\) together with the loss \(\ell\) is said to be \emph{eventually almost surely} (or \(\eas\)) predictable if there is a strategy such that the learner makes only \emph{finitely} many errors with probability \(1\) no matter what the underlying \(p \in \mathcal{P}\) is.

Our contributions in this paper establish a comprehensive theoretical
framework for the $\eas$-predictability. We summarized our main results in the follows.

\subsection{Summary of results}

Our first main result, Theorem~\ref{main1b}, demonstrates that the \(\eas\)-predictability of a model class is, in many cases, equivalent to a decomposition of the class into nested \(\eta\)-predictable subclasses. These are model classes where we can predict with a \emph{bounded} number of errors and with confidence \(\geq 1-\eta\) uniformly over all models in the class. Therefore, this result has the context of matching sample size to complexity—as the sample size increases, we address larger subsets of the model class. Put another way, this implies that such \(\eas\)-predictable estimators can be obtained through regularization based on the sample at hand.

We illustrate how Theorem~\ref{main1b} can be applied to transform seemingly impossible tasks, as posed by Cover~\cite{cover1973determining}, into essentially trivial problems in Example~\ref{eg:cover}. In Section~\ref{exam}, we demonstrate further applications of Theorem~\ref{main1b} across various contexts—hypothesis testing, risk domination, and online learning, to name a few. Notably, Theorem~\ref{main1b} enables us to rederive all the main results of Dembo and Peres~\cite{dembo1994topological} with straightforward and elementary proofs (see Section~\ref{hypotest} and Appendix~\ref{alternat}). Moreover, we partially resolve an open problem posed by Dembo and Peres~\cite{dembo1994topological} in Section~\ref{hypotest}.

Theorem~\ref{main1b}, while broadly applicable, does not always characterize \(\eas\)-predictability in the most general settings—indeed, we present counterexamples for the same. We then refine Theorem~\ref{main1b} for several special cases. Theorem~\ref{main1a} completely characterizes the \(\eas\)-predictability in the \emph{supervised setting}, i.e., a setting when the losses are observable from samples, using methods reminiscent of structural risk minimization. In Theorem~\ref{pureestimation}, we establish a different tight characterization of the \(\eas\)-predictability for \emph{pure estimation} problems with i.i.d. sampling and finite prediction domain. This result forms the basis of the almost sure hypothesis testing setup of Dembo and Peres~\cite{dembo1994topological} that we will study in detail in Section~\ref{hypotest}.

Finitely many errors do not imply that there exists a method, a stopping rule, that identifies when we are past the point of the final error. Indeed, in many cases, such a stopping rule is impossible. Section~\ref{ch2sec4} addresses this refinement with the notion of \(\eas\)-learnability. Informally, \(\eas\)-learnability of a model class requires that for any given confidence parameter \(\eta > 0\), we are able to find a prediction rule and stopping rule such that the probability of the prediction rule making errors after the stopping rule has ceased is upper bounded by \(\eta\). We provide a characterization of \(\eas\)-learnability in Theorem~\ref{main2b} through a notion of \emph{identifiability}. This characterization is tight when the processes are i.i.d.

Lastly, we introduce other natural variations of the \(\eas\)-predictability framework in Section~\ref{ch2sec5}, including a notion of weak \(\eas\)-predictability and the notion of predictability with finite expected loss. We show that predictability with finite expected loss implies \(\eas\)-predictability, and that \(\eas\)-predictability implies weak \(\eas\)-predictability. The converses do not hold, and these formulations are essentially different—we also provide examples that illustrate these nuances. We conclude the paper with several open problems that concern the exact characterizations of the above variations.

\subsection{Related work} Perhaps the most related non-uniform consistency setup in the learning theory literature is non-uniform PAC learning~\cite{blumer1989learnability, benedek1994nonuniform} and the concept of \emph{Structural Risk Minimization} (SRM)~\cite{vcnonuniform, vapnik2013nature}. In this setup, the goal is to obtain a \emph{distribution-free} sample complexity in the PAC model that could depend on the underlying hypothesis. It has been shown by~\cite{linial1991results, benedek1994nonuniform} that a hypothesis class can be learned in such a sense in the \emph{realizable} setting if and only if the class can be decomposed into a countable union of finite VC-dimensional classes. ~\cite{shawe1998structural} generalized the result to the \emph{agnostic} setup using SRM. There has also been research in the \emph{distribution-dependent} setting along this line, aiming to quantify the \emph{learning rates} of a hypothesis class with respect to the sample size~\cite{schuurmans1997characterizing, antos1998strong, bousquet2020theory}. In particular, ~\cite{bousquet2020theory} recently showed that under mild assumptions on the realizability of the underlying generating model, the learning rates can only asymptotically approach \(e^{-n}\), \(\frac{1}{n}\), or be arbitrarily slow. Our framework is also related to the learning of \emph{recursive functions} (a.k.a. inductive inference), introduced by Gold~\cite{gold1967language}. A typical setup starts with a set \(\mathcal{H}\) of functions that map \(\mathbb{N} \rightarrow \{0,1\}\). Nature fixes a function \(h\) from \(\mathcal{H}\) at the beginning. At each time step \(n\), the learner attempts to predict \(h(n)\) using the history \(h(1), \ldots, h(n-1)\) observed thus far. Nature then reveals the true value \(h(n)\) after the learner has made the prediction. The goal is a \emph{computable} learner that makes only \emph{finitely} many errors no matter what \(h \in \mathcal{H}\) is. It is shown by~\cite{barzdicnvs1972prediction} that a finite error computable predictor exists if and only if there exists a computable function \(g: \mathbb{N} \rightarrow \mathbb{N}\) such that the time complexity of \(\mathcal{H}\) is \emph{eventually dominated} by \(g\). We refer the reader to~\cite{gold1967language, barzdicnvs1972prediction, blum1975toward, zeugmann2008learning} for more results along this line. One may view the learning of recursive functions as a special case of our \(\eas\)-predictability where the underlying process is \emph{deterministic} but restricts the learner to be computable.

\section{Problem setup}
\label{ch2sec2}
Let $\mathcal{X}$ be a Polish space (separable completely metrizable topological space) with Borel $\sigma$-algebra and $\mathcal{P}$ be a collection of
probability measures over the cylinder $\sigma$-algebra over
$\mathcal{X}^{\infty}$. We consider a discrete time random process
$\X=\{X_n\}_{n\in \mathbb{N}^+}$ generated by sampling from a
probability law $p\in \mathcal{P}$. We will denote $X_i^j=(X_i,X_{i+1},\cdots,X_j)$ in the sequel.

Prediction is modeled as a measurable function $\Phi:\cX^*\to\cY$, where $\cX^*$
denotes the set of all finite strings of sequences from $\cX$, and
$\cY$ is the set of all predictions.  The \emph{loss} is a function
$\ell:\cP\times\mathcal{X}^*\times\mathcal{Y}\rightarrow \{0,1\}$ such that $\ell(p,\cdot,\cdot)$ is measurable for all $p\in \mathcal{P}$. We
consider the property we are estimating to be defined implicitly by
the subset of $\cP\times\mathcal{X}^*\times\mathcal{Y}$ where
$\ell=0$, and therefore, in a slight abuse of notation sometimes refer
to $\ell$ as a \emph{property} as well.

We consider the following game that proceeds in time indexed by
$\mathbb{N}^+$. The game has two parties: the Learner and Nature.
Nature chooses some model $p\in \cP$ to begin the game. At each time
step $n$, the Learner makes a prediction $Y_n=\Phi(X_1^{n-1})$
based on the current observation $X_1^{n-1}$ generated according to $p$.
Nature then generates $X_n$ based on $p$ and $X_1^{n-1}$.  

The Learner fails at step $n$ if $\ell(p, X_1^n,Y_n)=1$. The Learner
targets a strategy that minimizes the cumulative loss in the infinite
horizon, without knowledge of the model that Nature chooses at the
beginning.

The loss in general can be any function of the probability model in
addition to the sample observed, and our prediction on the sample.
When the loss depends on the probability model, there may be no direct
way to estimate the loss incurred at say, step $n$, from observations
of the sample $X_1^{n}$ even after the prediction $Y_n$ is made.
We call such setups the \emph{unsupervised setting} borrowing from
learning theory.  A special case is the \emph{supervised} setting,
where we define the loss to be a function from $\cX^*\times \cY$ to
$\sets{0,1}$.

\begin{definition}[$\eta$-predictability]
\label{etap}
A collection $(\cP,\ell)$ is
  $\eta$-predictable, if there exists a prediction
  rule $\Phi:\cX^*\to \cY$ and a sample size $n$ such that for all
  $p\in \cP$, 
\[
p
\Paren{
\sum_{i=n}^{\infty}\ell(p,X_1^i,\Phi(X_1^{i-1}))>0}
\le 
\eta,
\]
i.e. the probability that the learner makes errors after step $n$ is at
most $\eta$ uniformly over $\cP$.
\end{definition}

\begin{example}
\label{introexample1}
Let $\mathcal{P}^{\epsilon}$ be the class of all $i.i.d.$ Bernoulli processes with parameters in $\{p\in [0,1]: |p-\frac{1}{2}|\ge \epsilon\}\cup\{\frac{1}{2}\}$ for some absolute constant $\epsilon>0$. We define the loss to be $\ell(p,X_1^n,Y_n)=1\{Y_n\not=1\{p=\frac{1}{2}\}\}$, i.e., we would like to determine whether the underlying parameter equals $\frac{1}{2}$ or not. By Chernoff bound, we know that $(\mathcal{P}^{\epsilon},\ell)$ is $\eta$-predictable for all $\eta,\epsilon>0$, by predicting the (same) empirical mean of $X_1^N$ for all time step $>N$ where $N$ is any constant $\ge \frac{2\log 1/\eta}{\epsilon^2}$.
\end{example}

\begin{definition}[$\eas$-predictable]
\label{easp}
A collection $(\cP,\ell)$ is said to be eventually almost surely (e.a.s.)-predictable, if there exists a prediction rule $\Phi$, such that for all $p\in \cP$
\[
p
\left(
\sum_{n=1}^{\infty}\ell(p,X_1^n,\Phi(X_1^{n-1}))<\infty
\right)=1.
\]
\end{definition}

We need a technical definition that will help simplify notation
further.

\begin{definition} A \emph{nesting} of $\cP$ is a collection of subsets
of $\cP$, $\sets{\cP_i: i\ge 1}$ such that $\cP_1\subset \cP_2\subset\ldots$
and $\bigcup_{i\ge1} \cP_i = \cP$.
\end{definition}

The following lemma characterize immediate connections between
the above definitions.
\begin{lemma}
\label{union2eas}
Let $\mathcal{P}$ be a collection of models, $\{\mathcal{P}_i,i\ge 1\}$ be a nesting of $\mathcal{P}$. If for all $\eta>0$ and $i\in \mathbb{N}^+$, $(\mathcal{P}_i,\ell)$ is $\eta$-predictable. Then $(\mathcal{P},\ell)$ is $\eas$-predictable.
\end{lemma}
\begin{proof}
  From the definition of $\eta$-predictability, we can choose an
  increasing sequence $\sets{b_i, i\ge 1}$, and predictors $\Phi_i$
  for $\cP_i$ respectively as follows.
  For all $i$ and for all $p\in\cP_i$, the probability
  $\Phi_i$ makes errors after step $b_i$ is at most $2^{-i}$.

  The predictor $\Phi$ is then constructed from $\sets{\Phi_i, i\ge1}$
  as follows: use predictor $\Phi_i$ when the length $T$ of the
  observed sample satisfies $b_i\le T<b_{i+1}$.

  Let $p\in\mathcal{P}_k\subset\cP$. Because the collections $\cP_i$ are
  nested, for all $i\ge k$, $p\in\cP_i$. During the phase $\Phi$
  coincides with $\Phi_i$, the probability of $\Phi$ making an error
  is $\le 2^{-i}$. The result follows using the Borel-Cantelli lemma.
\end{proof}

\begin{example}
\label{introexample2}
Let $\mathcal{P}$ be the class of all \iid Bernoulli processes with parameters in $[0,1]$ and $\ell$ is the same loss as in Example~\ref{introexample1}. We claim that $(\mathcal{P},\ell)$ is $\eas$-predictable using Lemma~\ref{union2eas}. For any $i\ge 1$, we define $\mathcal{P}_i=\mathcal{P}^{1/i}$, where $\mathcal{P}^{1/i}$ is defined as in Example~\ref{introexample1} with $\epsilon=\frac{1}{i}$. Clearly, $\{\mathcal{P}_i,i\ge 1\}$ forms a nesting of $\mathcal{P}$. By Example~\ref{introexample1}, we also know that each $\mathcal{P}_i$ is $\eta$-predictable for all $\eta>0$, the claim follows.
\end{example}

\paragraph*{Refined notions of nestings} Motivated by the above lemma, we introduce the following two refinements on the notion of nestings.

\begin{definition}[Universal nestings]
A nesting $\{\mathcal{P}_i,i\ge 1\}$ of $\mathcal{P}$ is a \emph{universal} nesting w.r.t loss $\ell$, if for all $\eta>0$ and $i\ge 1$, $(\mathcal{P}_i,\ell)$ is $\eta$-predictable.
\end{definition}

\begin{definition}[$\eta$-nestings]
For any $\eta>0$, a nesting $\{\mathcal{P}_i^{\eta},i\ge 1\}$ of $\mathcal{P}$ is an \emph{$\eta$-nesting} w.r.t. loss $\ell$, if for all $i\ge 1$, $(\mathcal{P}_i^{\eta},\ell)$ is $\eta$-predictable.
\end{definition}

Clearly, if a class $\mathcal{P}$ has a universal nesting then it has $\eta$-nesting for all $\eta>0$. However, the converse is not true, as we will see in Section~\ref{sec:easp}. Indeed, as we will see in Theorem~\ref{main1b} that these two concepts will form the basis for characterizing $\eas$-predictability.

The following technical lemma will be useful in our following proofs.

\begin{lemma}
\label{decomeq}
Let $\mathcal{P}$ be a collection of probability measures and let
$\ell$ be a loss function. Suppose 
$\{\mathcal{P}_i:i\ge 1\}$ is a nesting of $\cP$ such that for all
$i\ge 1$, $(\cP_i,\ell)$ is $\eta$-predictable for some $\eta>0$.

Then there exists a relabeling of the sets in a nesting
$\{\cP'_{i}:i\ge 1\}$ of $\cP$ such that $(\cP'_{i},\ell)$ is
$\eta$-predictable with sample size $i$.  Namely, there is a
predictor $\Phi_i$ such that for all $p\in\cP'_{i}$, the probability
$\Phi_i$ incurs non-zero $\ell$-loss on samples with size larger than
$i$ is $\le \eta$.
\end{lemma}
\begin{proof}
  Since $\cP_i$ is $\eta$-predictable, there exists a number $n_i'$
  and a predictor $\Phi_i$ such that for all $p\in\cP_i$, the
  probability $\Phi_i$ incurs non-zero $\ell-$loss on samples larger
  than $n_i'$ is $\le \eta$. We can therefore choose an increasing
  sequence $\sets{n_i, i\ge 1}$ with $n_i\ge n_i'$. Note that each
  $\cP_i$ is $\eta$-predictable with sample size $n_i$. For
  $n_k\le i<n_{k+1}$, set $\cP_i'=\cP_k$. $\sets{\cP_i': i\ge 1}$ is
  the desired nesting and the lemma follows.
\end{proof}

\section{Characterization of $\eas$-predictability}
\label{sec:easp}
We first provide a general characterization of the $\eas$-predictability without any assumption on the class $\mathcal{P}$ and loss $\ell$.

\begin{theorem} 
\label{main1b}
Consider a collection $\mathcal{P}$ with a loss
$\ell:\cP\times\cX^*\times \cY\to \sets{0,1}$. $(\cP,\ell)$ is $\eas$-predictable if there exists a nesting
$\sets{\cP_i : i\ge 1}$ of $\cP$ such that for all $\eta>0$, $i\ge 1$, $(\cP_{i},\ell)$ is
$\eta$-predictable.

Conversely, if $(\cP,\ell)$ is $\eas$-predictable, then for all $\eta>0$, there is a
nesting $\sets{\cP_i^\eta : i\ge 1}$ of $\cP$ such that for all $i\ge 1$, $(\cP_i^\eta,\ell)$
is $\eta$-predictable.
\end{theorem}
\begin{proof}
The sufficiency follows directly from Lemma~\ref{union2eas}. We now prove the necessary condition. Suppose $\cP$ is $\eas$-predictable, we show that $\cP$ can be
  decomposed into a nesting of $\eta$-predictable collections. By
  Definition~\ref{easp}, there exists a predictor $\Phi$ such that for
  all $p\in\cP$, $\Phi$ makes finitely many errors with probability
  1. For $\eta>0$, we define
  $$\mathcal{P}_{n}^\eta = \{ p\in \mathcal{P} \mid p(\Phi\text{ makes
    errors after time }n)<\eta \},$$ so that for all $n$, $\cP_n^\eta$
  is $\eta$-predictable by definition.  Further, by definition,
  $\forall n\in\mathbb{N}^+$, $\cP_n^\eta\subset \cP_{n+1}^\eta$. To see
  that the union of $\cP^\eta_n$ over all $n$ is $\cP$, for all $p\in\mathcal{P}$, we consider the
  event
  $$
  A_k^p =
  \sets{
    X_1^{\infty}
    \mid
    \sum_{n=k}^{\infty} \ell(p,X_1^n,\Phi(X_1^{n-1}))>0}.
  $$
  We have $p(A_k^p)\rightarrow 0$ as
  $k\rightarrow \infty$ by $\eas$-predictability of $\Phi$. Therefore, there must be some $k$ such that
  $p(A_k^p)<\eta$, and for such a number $k$ we have
  $p\in \mathcal{P}_k^\eta$. Therefore,
  $\mathcal{P}=\bigcup_{n\in \mathbb{N}^+}\cP_n^\eta$. This completes the proof.
\end{proof}

Theorem~\ref{main1b} shows that to prove a collection $(\mathcal{P},\ell)$ is $\eas$-predictable, it is sufficient to find a \emph{universal nesting} of $\mathcal{P}$. To prove that a collection $(\mathcal{P},\ell)$ is not $\eas$-predictable it is sufficient to show that for some $\eta>0$ there is no \emph{$\eta$-nesting} of $\mathcal{P}$. 

We provide an example below, which illustrates how Theorem~\ref{main1b} can be used to derive the result of Cover~\cite{cover1973determining}. See also Section~\ref{exam} for more applications in different contexts.
\begin{example}
\label{eg:cover}
  The task is to predict whether the
parameter $p$ of an \iid Bernoulli$(p)$ process is rational or not
using samples from it.

Therefore our predictor
$$\Phi: \sets{0,1}^*\to \sets{ \text{rational}, \text{ irrational}
}.$$In~\citep{cover1973determining}, Cover showed a scheme that
predicted accurately with only finitely many errors for all rational
sources, and for a set of irrationals with Lebesgue measure 1. Here we
show a more transparent version of Cover's proof as well as subsequent
refinements in~\cite{koplowitz1995cover} using Theorem~\ref{main1b}
above and an argument evocative of regularization.

Define the loss $\ell(p, X_1^n,Y_n)=0$ iff $Y_n$ gives the correct irrationality of $p$. Note that the setting is 
what we would call the ``unsupervised'' case and that there is no way
to judge if our predictions thus far are right or wrong.

Let $r_1,r_2,\cdots$ be an enumeration of rational numbers in
$[0,1]$. 
Let $B(p,\epsilon)$ be the set of numbers in $[0,1]$ whose
 $L_1$ distance from $p$ is $< \epsilon$. For all $k$, let 
\[
\mathcal{S}_k
=
\Paren{
[0,1]
\backslash
\bigcup_{i=1}^{\infty}B(r_i,\frac{1}{k2^i})}
\cup
\{r_1,\cdots,r_k\}
\]
be the set that excludes a ball centered on each rational number, but
throws back in the first $k$ rational numbers. Note that the Lebesgue
measure of $\cS_k$ is $\ge 1-\frac1k$. Now $\mathcal{S}_k$ contains
exactly $k$ rational numbers, the rest being irrational. Moreover,
$S_k$ contains no irrational number within distance $ \le 2^{-k}/k$
from any of the included rationals. Hence, the set $\cB_k$ of
Bernoulli processes with parameters in $\mathcal{S}_k$ is
$\eta$-predictable for all $\eta>0$.

From Theorem~\ref{main1b}, we can conclude that the collection $\cB
\ed\bigcup_{k\in \mathbb{N}}\mathcal{B}_k$ is $\eas$-predictable.
Note that every rational number belongs to $\mathcal{S}=\bigcup_{k\in \mathbb{N}}\mathcal{S}_k$, and the set of irrational numbers in $\mathcal{S}$ 
has Lebesgue measure 1, proving~\citep[Theorem 1]{cover1973determining}.

Conversely, let $\mathcal{S}\subset[0,1]$ and $\mathcal{B}$ be the
Bernoulli variables with parameters in $\mathcal{S}$.  We show that if
$\mathcal{B}$ is $\eas$-predictable for rationality of the underlying
parameter, then $\mathcal{S}=\bigcup_{k\in \mathbb{N}}\mathcal{S}_k$
such that $S_k\subset S_{k+1}$ and
\[
\inf\{|r-x|:  
r,x\in \mathcal{S}_k
\text{, }r\text{ is rational, }x\text{ is irrational}\}>0.
\] 
Since $\mathcal{B}$ is $\eas$-predictable, Theorem~\ref{main1b} yields
that for any $\eta>0$, the collection $\mathcal{B}$ can be decomposed
as $\mathcal{B}=\bigcup_{k}\mathcal{B}_{k}$ where each
$\mathcal{B}_{k}$ is $\eta$-predictable and
$\forall k,~\mathcal{B}_k\subset\mathcal{B}_{k+1}$. Let
$\mathcal{S}_k$ be the set of parameters of the sources in
$\mathcal{B}_k$.  Intuitively,
$\eta$-predictability of $\mathcal{B}_k$ implies that we must have
$$
\inf \sets{|u-v|: u,v\in \mathcal{S}_{k}\text{, } u\text{ rational,
  }v\text{ irrational}} >0,$$
or else we would not be able to universally
attest to rationality with confidence $\ge 1-\eta$ using a bounded number of
samples. See Lemma~\ref{coverl2} for a formal proof.

Suppose we want $\mathcal{S}$ to contain all rational numbers in
$[0,1]$. Then it follows (see Lemma~\ref{coverl3}) that the subset of irrational
numbers of $\mathcal{S}_k$ must be nowhere dense. Therefore, the set
of irrationals in $\mathcal{S}$ is a meager or Baire first category
set~\cite[Chapter 2.1]{rudin2006functional}, proving~\cite[Theorem 2]{koplowitz1995cover}.
\end{example}

Note that the necessary and sufficient conditions in Theorem~\ref{main1b} do not necessarily match in general. Indeed, the following example shows that the necessary (respectively sufficient) condition in Theorem~\ref{main1b} is not sufficient (respectively necessary).

\begin{example}
\label{etanot2eas}
  Let $\mathcal{P}$ be a
  class of binary (taking values 0 or 1) random processes that
  converge to either 0 or 1 in probability. Formally, $\cP$ is the
  collection of all probability measures $p_b$, $b\in\sets{0,1}$,
  defined on the Borel $\sigma$-algebra of $\{0,1\}^\infty$,
  that satisfies
  $$\lim_{n\to\infty} p_b(X_{n}=b)=1.$$
  The task of the prediction $Y_n$ is to predict the parameter $b$
  associated with the process, and takes values in $\{0,1\}$. The loss
  $\ell$ associated with the prediction is defined as
  $$\ell(p_b,X_1^n,Y_n)=1\{Y_n\not=b\}.$$

We now show that the condition deemed necessary in
Theorem~\ref{main1b} holds for $(\mathcal{P},\ell)$. To see this, let
$\mathcal{P}_i^{\eta}$ be the class of processes $p_b\in\mathcal{P}$ such
that for all $n\ge i$ 
\[
  p_b(X_{n}=b ) \ge 1-\eta.
\]
The $\eta$-predictability of $\mathcal{P}_i^{\eta}$ follows because
$p_b(X_i=b) \ge 1-\eta$, and a predictor that predicts $X_i$ for all time
steps $\ge i$ will incur loss $0$ past time step $i$ whenever $X_i=b$.

We show that the collection $(\mathcal{P},\ell)$ is, however,
\emph{not} $\eas$-predictable. To see this, suppose such a prediction
rule $\Phi$ exists. We first observe that there exists a number $N(m)$
such that for all finite binary sequences $x_1,\cdots,x_m$ of length
$m$ and all $b\in\{0,1\}$
\begin{equation}
  \label{eq:dummy}
  \Phi(x_1,\cdots,x_m,b,\cdots,b)=b,
\end{equation}
whenever the number of $b$'s is larger than $N(m)$. This holds because
each of the $2\cdot 2^m$ semi-infinite strings $x_1,\cdots, x_m, b,\ldots$
corresponds to a process in $\cP$ that assigns probability 1 to that
string. If~\eqref{eq:dummy} did not hold, $\Phi$ would make an infinite
number of errors on one of these processes, contradicting the
$\eas$-predictability of $\Phi$ on $(\mathcal{P},\ell)$.

We now construct the following process $p_0$ in $\cP$ that will break
$\Phi$. Let $M_0=0$, $M_1=N(0)+1$, and recursively define
$M_n=N(M_{n-1})+1$. The process $p_0$ is partitioned into independent sample
blocks, where the $n$th block ranges from $X_{M_n+1}$ to $X_{M_{n+1}}$
such that $X_{M_{n}+1}=X_{M_{n}+2}=\cdots=X_{M_{n+1}}$ and
$$p_0(X_{M_{n}+1}=0)=1-\frac{1}{n}.$$
Let $A_n$ be the event that $X_{M_{n}+1}=1$. We have $A_n$ happens infinitely often almost surely by the converse of the Borel-Cantelli lemma, since $\sum_np_0(A_n)=\sum\frac{1}{n}=\infty$ and $A_n$'s are independent. By construction, $\Phi$ makes errors in sample block $n$ if $A_n$ happens, hence $\Phi$ makes infinitely many errors almost surely. But clearly $p_0\in \mathcal{P}$, contradicting the $\eas$-predictability of $\Phi$.

With similar construction, we can also show that a collection $(\mathcal{P},\ell)$ is $\eas$-predictable does not imply the existence of a universal nesting of $\mathcal{P}$. See Example~\ref{easnot2universal} for a detailed construction.
\end{example}

While Theorem~\ref{main1b} may look rather innocuous, it offers a
partial resolution to an open problem in~\cite{dembo1994topological}.
Let $\mathcal{H}_1$ and $\mathcal{H}_2$ be disjoint classes of
distributions over $\mathbb{R}^d$. Let $\mathcal{H}$ be the class of
all $\iid$ random processes with marginal distributions from
$\mathcal{H}_1\cup \mathcal{H}_2$. Dembo and Peres \cite{dembo1994topological}
considered the problem of identifying whether the marginal of a \iid
random process in $\mathcal{H}$ comes from $\mathcal{H}_1$ or
$\mathcal{H}_2$ by observing samples from it. The prediction domain
now is $\mathcal{Y}=\{1,2\}$ and loss is
$\ell^{HT}(p,X_1^n,Y_n)=1\{p\not\in
\mathcal{H}_{Y_n}\}$. Dembo and Peres \cite{dembo1994topological} showed that if the
distributions in $\mathcal{H}_1\cup \mathcal{H}_2$ have densities
\emph{and} there exists some $r>1$ such that the $r$-th norm of the
densities are finite, then $(\mathcal{H},\ellht)$ is $\eas$-predictable iff
the distributions in $\mathcal{H}_1$ and $\mathcal{H}_2$ are
\emph{$F_{\sigma}$-separable} (see Section~\ref{hypotest} for
definition) under any metric consistent with weak convergence
topology. 

Dembo and Peres \cite{dembo1994topological} asked whether the condition $r>1$ can be
removed. We give a positive answer to this problem as follows.
\begin{corollary}
\label{demboopen}
Suppose there exists a monotonically increasing function
$G:\mathbb{R}^+\rightarrow\mathbb{R}^+$ 
with $\lim_{x\rightarrow\infty}G(x)=\infty$ such that for
any distribution $p\in\mathcal{H}_1\cup\mathcal{H}_2$ with density
$f_p(x)$, we have $\mathbb{E}_{X\sim p}[G(f_p(X))]<\infty$. Then
$(\mathcal{H}, \ellht)$ is $\eas$-predictable iff the distributions in
$\mathcal{H}_1$ and $\mathcal{H}_2$ are $F_{\sigma}$-separable under
weak convergence topology.
\end{corollary}
The proof and a discussion of the Corollary above is left to Section~\ref{hypotest}.

\section{Specialized settings}
\label{s:specialized}
While the necessary and sufficient conditions in Theorem~\ref{main1b} do not match up in general as shown above, the characterization of $\eas$-predictability can be tightened in several natural settings. We split our further analysis into two settings---one where the loss can be observed (the supervised setting) and where the loss cannot be observed (unsupervised setting). We provide tight characterizations for the supervised setting, and for $\iid$ generated data when the loss is only a function of the source and the prediction in the unsupervised setting. 

In this section, the distinction between $\eta$-nesting and universal nesting will become clear as well---we will also encounter examples that allow $\eta$-nestings for all $\eta>0$, but which lack a universal nesting even in the supervised setting.

\subsection{Supervised setting}
\label{ch2sec3.3.1}
We now characterize $\eas$-predictability in \emph{supervised} setting, i.e. the loss $\ell:\mathcal{X}^*\times \mathcal{Y}\rightarrow \{0,1\}$ is independent of the underlying source $p$ (or equivalently the loss can be gauged from the samples). We show that the existence of $\eta$-nesting for all $\eta>0$ is necessary and sufficient to achieve $\eas$-predictability in the supervised case. As we have demonstrated in Example~\ref{etanot2eas} the existence of $\eta$-nesting for all $\eta>0$ does not imply $\eas$-predictability in the most general case. Theorem~\ref{main1a} shows that the implication holds when the loss is supervised.

\begin{theorem}
\label{main1a}
Consider a collection $\mathcal{P}$ with a loss
$\ell:\cX^*\times \cY\to \sets{0,1}$ (i.e. the supervised
setting). $(\cP,\ell)$ is $\eas$-predictable iff for all $\eta>0$, there exists
a $\eta$-nesting of $\mathcal{P}$.
\end{theorem}
\begin{proof}  
  The necessity follows directly from Theorem~\ref{main1b}.
  We now prove that the existence of $\eta$-nesting for all $\eta>0$ is also sufficient. Suppose that for all
  $j\in \mathbb{N}$, there exists a nesting
  $\sets{\mathcal{P}_n^j:n\ge1}$ of $\cP$ such that $\mathcal{P}_n^j$
  is $2^{-j}$-predictable for all $n\ge 1$. Furthermore, from Lemma~\ref{decomeq}, we
  can choose a decomposition such that $\mathcal{P}_n^j$ is $2^{-j}$
  predictable with sample size $n$. Therefore, there exist predictors
  $\Phi_{n,j}$ such that for all $p\in \cP_n^j$
  $$
  p(\Phi_{n,j}\text{ makes errors after time } n) \le 2^{-j}.$$

  We construct a predictor $\Phi$ for $\cP$ as follows. At each time
  step $T$, let $I(n,j)$ be the indicator that $\Phi_{n,j}$ makes no
  error on $X_1^{T-1}$ after time $n$. Let
 \begin{equation}
 \label{thm1eqindex}
 (k,i)=\argmin_{(n,j)\in \mathbb{N}\times \mathbb{N}}\{j+n\mid I(n,j)=1\}.
\end{equation}  
  The prediction is defined as $\Phi(X_1^{T-1})=\Phi_{k,i}(X_1^{T-1})$.
  
We claim that the predictor $\Phi$ will make only finitely many errors with probability $1$ for all models in $\mathcal{P}$. Fix some $p\in \cP$. Let $n_j=\min\{n\mid p\in \cP_{n}^j\}$. Define the event 
\[
A_j=\{\Phi_{n_j,j}\text{ makes errors after time step }n_j\}.
\]
We have
$\sum_{j=1}^{\infty}p(A_j)\le\sum_{j=1}^{\infty}2^{-j}<\infty$.

Therefore, the Borel-Cantelli lemma implies that there is a set with
probability 1, such that on every semi-infinite sequence $X_1^\infty$ in
that set, there is a $J$ such that $\Phi_{n_J,J}$ makes no errors after
step $n_J$.
By construction of $\Phi$, for $X_1^\infty$ in the set of
probability 1 above, we will therefore never choose an estimator
$\Phi_{n,j}$ with $n+j>n_J+J$ in step~\eqref{thm1eqindex}. If some
$\Phi_{n,j}$ with $n+j\le n_J+J$ makes infinitely many errors, it will
no longer appear in the feasible set in~\eqref{thm1eqindex} after some
time step $T\ge n$. Since there are only finitely many predictors
$\Phi_{n,j}$ with $n+j\le n_J+J$, the procedure will eventually choose
some predictor that makes finitely errors.
\end{proof}

\begin{remark}
 Note that, the approach in the proof of Theorem~\ref{main1a} can be viewed as \emph{regularization}. The measure of complexity here is $n+j$. Here $j$ equals the negative log of the confidence probability of the $2^{-j}-$nesting $\{{\cal P}_n^j,n\ge 1\}$ and $n$ denotes that we restrict the class to ${\cal P}_n^j$, i.e., it is a measure of the complexity of the class considered. The smaller $n+j$ is, the less the complexity the
  $\mathcal{P}_n^j$ has. The indicator $I(n,j)$ is interpreted as consistency of the sample with class $\mathcal{P}_n^j$. At each step, equation~(\ref{thm1eqindex}) selects a subclass with minimum complexity that is consistent with current sample.
\end{remark}

While $\eas$-predictability is completely characterized by the existence of $\eta$-nesting for all $\eta>0$ in the supervised setting, $\eas$-predictability does not imply the existence of a \emph{universal nesting} as the following example demonstrates. A corollary of this is that the existence of $\eta$-nesting for all $\eta$ is not equivalent to the existence of a universal nesting even in the supervised setting---the two are different concepts.

\begin{example}
\label{easnot2universalsuper}

Let $\mathcal{P}$ be the class of the following processes. For any infinite sequence $M_1<M_2<\cdots <M_n<\cdots\in \mathbb{N}^{\infty}$, we define process $p^M$ over $\{0,1\}^{\infty}$ as follows:
$$X_{M_n}\sim \text{Bernoulli}\left(\frac{1}{2}\right)~\text{independently for all }n\in \mathbb{N},$$
and $X_t=0$ for all other time step $t$. The loss $\ell$ is defined as follows:
$$\ell(p^M,X_1^t,Y_t)=0\text{ iff there exists at least one }1 \text{ in }X_1^t.$$
Note that the loss only depends on the sample but not the underlying source and predictions. Clearly, the loss is observable from the sample, thus the setup is supervised. Furthermore, $(\mathcal{P},\ell)$ is $\eas$-predictable since for any source $p^M$, $p^M(\exists n\ge 1,~X_{M_n}=1)=1$. 

Let $\mathcal{P}_1\subset\mathcal{P}_2\subset\cdots\subset\mathcal{P}$ be an arbitrary nesting of $\mathcal{P}$. We show that $\{\mathcal{P}_k,~k\ge 1\}$ cannot be a universal nesting of $\mathcal{P}$ by contradiction.

Suppose the above were a universal nesting, so each ${\cal P}_k$ in the nesting is $\eta$-predictable for all $\eta>0$. For all $k$, let $B_k$ be the sample size for $\mathcal{P}_k$ to achieve $\frac{1}{2^{k+1}}$-predictability. Now let $p^B\in \mathcal{P}$ be the source that is associated with sequence $\{B_k+1\}_{k\in \mathbb{N}^+}$. 

We claim that $p^B\not\in \mathcal{P}_k$ for all $k\in \mathbb{N}^{+}$. To see this, note that $p^B(\forall i\le k,~X_{B_i+1}=0)=\frac{1}{2^k}$ which in turn implies that with probability at least $\frac{1}{2^k}$, we make an error at step $B_k+1$. If for any $k$, $p^B$ were in ${\cal P}_k$, $\mathcal{P}_k$ would not be $\frac{1}{2^{k+1}}$-predictable. Therefore  $p^B\not\in \bigcup _{k}\mathcal{P}_k$. However, we also know that $p^B\in\mathcal{P}$, meaning that $\{\mathcal{P}_k,k\ge 1\}$ cannot be a nesting of $\mathcal{P}$.
\end{example}

\subsection{Unsupervised setting}
While $\eta$-nestings (for all $\eta$) characterizes the supervised setting, it does not carry over for the unsupervised case in the most general cases. In this subsection, we introduce several special cases for the unsupervised setting where tight characterizations are available.

We first prove the following theorem, which shows that if $\mathcal{P}$ is \emph{countable} then existence of universal nesting $\Leftrightarrow$ $\eas$-predictability $\Leftrightarrow$ existence of $\eta$-nesting for all $\eta>0$.

\begin{theorem}
Let $\mathcal{P}$ be a \emph{countable} class and $\ell$ be an arbitrary loss. Then there exists a universal nesting of $\mathcal{P}$ $\Leftrightarrow$ $(\mathcal{P},\ell)$ is $\eas$-predictable $\Leftrightarrow$ there exist $\eta$-nesting of $\mathcal{P}$ for all $\eta>0$.
\end{theorem}
\begin{proof}
Let $p_1,p_2,\cdots$ be an arbitrary enumeration of $\mathcal{P}$.

We first show that if $(\mathcal{P},\ell)$ is $\eas$-predictable then there is a universal nesting for $\mathcal{P}$. Define $\mathcal{P}_k=\{p_1,\cdots,p_k\}$. We show that $(\mathcal{P}_k,\ell)$ is $\eta$-predictable for all $\eta>0$ and $k\ge 1$. To see this, let $\Phi$ be the $\eas$-prediction rule for $(\mathcal{P},\ell)$. For any given $\eta>0$ and $i\ge 1$, there exists a number $N_{i,\eta}$ such that $\Phi$ makes errors after step $N_{i,\eta}$ w.p. $\le \eta$ when $p_i$ is in force. Now, for any $\mathcal{P}_k$, we know that $\Phi$ makes no errors after step $\max_{1\le i\le k}\{N_{i,\eta}\}$ w.p. $\ge 1-\eta$ no matter what source $\in\mathcal{P}_k$ is in force. Therefore $(\mathcal{P}_k,\ell)$ is $\eta$-predictable for all $\eta>0$ and $k\ge 1$.

We now prove that if there exists $\eta$-nesting of $\mathcal{P}$ for all $\eta>0$ then $(\mathcal{P},\ell)$ is $\eas$-predictable. To see this, let $\{\mathcal{P}_k^j,k\ge 1\}$ be the $2^{-j}$-nesting of $\mathcal{P}$. By Lemma~\ref{decomeq}, we can assume that there exists prediction rule $\Phi_k^j$ such that $\Phi_k^j$ makes no errors after step $k$ w.p. $\ge 1-2^{-j}$ for all $p\in \mathcal{P}_k^j$. We can assume, w.l.o.g., that $\mathcal{P}_k^j$ contains only the sources in $\{p_1,\cdots,p_k\}$. For all $j\ge 1$, let $K_j$ be the minimum index of class in $\{\mathcal{P}_k^j\}$ such that $\mathcal{P}_{K_j}^j$ contains all sources in $\{p_1,\cdots,p_j\}$. We now construct the $\eas$-prediction rule as follows. We partition the prediction into phases, initially at phase $0$. We are in phase $j\ge 1$ if the time step $n$ satisfies $K_j\le n<K_{j+1}$. At phase $j\ge 1$ we use $\Phi_{K_j}^j$ to make the prediction and predict arbitrarily at phase $0$.

Suppose the underlying source is $p_t$, we know that $\mathcal{P}_{K_j}^j$ contains $p_t$ for all $j\ge t$. Meaning that, the probability we make errors at phase $j\ge t$ is at most $2^{-j}$. The theorem follows by the Borel-Cantelli lemma.
\end{proof}

We now introduce another unsupervised setting where we obtain a full characterization of $\eas$-predictability---the \emph{pure estimation} setting with \emph{finite} prediction domain and $\iid$ sampling, where the loss is a function of only the source and prediction. This includes, e.g., the almost sure hypothesis testing framework of Dembo and Peres~\cite{dembo1994topological}.

\begin{theorem}
\label{pureestimation}
  Let $\mathcal{P}$ be a model collection of $\iid$ measures over
  $\mathcal{X}^{\infty}$, $\ell$ is a loss function that only depends
  on the prediction and underlying source but not samples,
  i.e. $\ell:\mathcal{P}\times \mathcal{Y}\rightarrow \{0,1\}$. If
  $|\mathcal{Y}|$ is finite, then $(\mathcal{P},\ell)$ is
  $\eas$-predictable iff there exists a universal nesting
  $\{\mathcal{P}_i,~i\ge 1\}$ of $\mathcal{P}$ such that
  for all $\eta>0$, $(\mathcal{P}_i,\ell)$ is $\eta$-predictable.
\end{theorem}
\begin{proof}
  Applying Theorem~\ref{main1b} to $(\cP,\ell)$, we know that if the
  universal nesting exists, then $(\cP,\ell)$ is $\eas$-predictable.
  Theorem~\ref{main1b} also guarantees that if $(\cP,\ell)$ is
  $\eas$-predictable, then for all $\eta>0$, there is a nesting
  $\sets{\cP^{\eta}_i}$ where $(\cP^\eta_i,\ell)$ is $\eta$-predictable.
  
We prove the theorem by showing that if there exists some $\eta>0$
for which there is a nesting $\sets{\cP_i}$ where
$(\mathcal{P}_i,\ell)$ is $\eta$-predictable, then this nesting is
also universal for all $\eta>0$, \emph{i.e.} $(\mathcal{P}_i,\ell)$ is also
$\eta$-predictable for all $\eta>0$.

To do so, we show that if $(\mathcal{P}_i,\ell)$ is
$(\frac{1}{|\mathcal{Y}|}-\epsilon)$-predictable for any $\epsilon>0$
then it is $\eta$-predictable for all $\eta>0$. Suppose $N$ is the sample
size such that a predictor $\Phi$ makes no errors past $N$ with probability
$\ge 1- \frac1{|\mathcal{Y}|}+\epsilon$.
For any $\eta< \frac1{|\mathcal{Y}|}-\epsilon$, 
let $M=\frac{2\log(\eta/2)}{\epsilon^2}$ and
consider a sample of size $MN$. We split this sample into $M$ blocks
of size $N$ each, and apply the predictor $\Phi$ to each block,
obtaining a prediction $Y_i\in\mathcal{Y}$ for the $i'$th block. Our
prediction for the sample of size $MN$ is an element of $\cY$ that
repeats most often in $Y_1,\cdots, Y_M$.

By Hoeffding bound with probability $\ge 1-\eta$, any element of $\cY$
that incurs loss 1 against the underlying distribution will appear at
most $\frac{1}{|\mathcal{Y}|}-\epsilon/2$ frequency among
$Y_1,\ldots, Y_M$. But at least one element of $\cY$ appears more than
$1/|\mathcal{Y}|$ frequency among $Y_1,\ldots, Y_M$, and any such an element must
incur 0 loss. The theorem follows.
\end{proof}

\begin{remark}
Note that the finiteness of the prediction domain $\mathcal{Y}$ is essential for the proof of the above Theorem to work. We leave it as an open problem to determine if the theorem above can be extended to countable domain $\mathcal{Y}$ and general loss $\ell:\mathcal{P}\times \mathcal{Y}\rightarrow\{0,1\}$.
\end{remark}

\section{Capturing the final error}
\label{ch2sec4}
While $\eas$-predictability is an attractive setup when considering
rich model classes, we would like to see if a predictor that makes
finitely many errors has finished making the errors.  Namely, can we
obtain a stopping rule that identifies the last error? Recall that a stopping
rule is a function $\tau : \mathcal{X}^*\rightarrow \{0,1\}$, such
that $\tau(\text{y})\le \tau(\text{x})$ if $\text{y}$ is a prefix of
$\text{x}$. We interpreted $\tau=0$ as the waiting period, and
$\tau=1$ as the rule has stopped waiting.

\begin{definition}[$\eas$-learnable]
\label{easl}
A collection $(\mathcal{P},\ell)$ is said to be eventually almost surely (e.a.s.)-learnable, if for any $\eta>0$, there exists a universal prediction rule $\Phi_{\eta}$ together with a stopping rule $\tau_{\eta}$, such that for all $p\in \cP$
$$
p
\left(
  \sum_{n=1}^{\infty}
  \ell( p, X_1^n, \Phi_{\eta}(X_1^{n-1}))
  \tau(X_1^{n-1})
  >0
\right)<\eta,$$
and
$$p
\Paren{
  \lim_{n\to\infty}\tau_{\eta}(X_1^n)=1}
=1.$$
\end{definition}

Clearly, $\eas$-learnability implies $\eas$-predictability. 

\begin{theorem}
\label{predict2learn}
Any $\eas$-learnable $(\mathcal{P},\ell)$ is $\eas$-predictable.
\end{theorem}
\begin{proof}
Suppose $(\mathcal{P},\ell)$ is e.a.s.-learnable. Then for each $i$,
we let $\Phi_i$ and $\tau_i$ be the predictor and stopping rule pair
respectively that $\eas$-learns $(\mathcal{P},\ell)$ with $\eta=1/2^i$.  By definition, we
have that the probability $\Phi_i$ makes an error after $\tau_i$ stops
(i.e. $\tau_i=1$) is $\le \frac1{2^i}$. Let $\Phi_0$ be an arbitrary
predictor.

Now, there are countably many stopping rules (one for each natural
number $i\ge0$) and each such rule stops at a finite time with
probability $1$, we conclude that with probability $1$ all of them
would have stopped simultaneously at some finite time by a union
bound.

We initialize $t=1$ ($t$ will stand for the stage). As we see more of
the sample, at any stage $t$, we predict using the prediction rule
$\Phi_{t-1}$, till $\tau_t$ halts (i.e. $\tau_t=1$). At that point, we
move to stage $t+1$. For $t\ge2$, the probability of making an error
in stage $t$ is $\le 2^{-t}$.  Invoking the Borel-Cantelli lemma, we
conclude that we make errors in finitely many stages almost surely,
and the Theorem follows.
\end{proof}

However, $\eas$-predictability does not imply $\eas$-learnability.

\begin{example}
\label{predictnot2learn}
Let $\mathcal{P}$ be the class of all $\iid$ Bernoulli random processes with parameters in $[0,1]$. For any $p\in \mathcal{P}$, we would like to determine if the parameter of the process equals 1/2 or not (prediction is 1 if parameter is 1/2, 0 else), namely, the loss is $\ell(p,X_1^n,Y_n)=1\{Y_n=1\{p=\frac{1}{2}\}\}$, where in a slight abuse of notation, we use $p$ to denote both the iid Bernoulli source and its parameter. In Example~\ref{introexample2}, we have shown that $(\mathcal{P},\ell)$ is $\eas$-predictable. We now show that $(\mathcal{P},\ell)$ is not $\eas$-learnable. 

Suppose otherwise, let $\Phi$, $\tau$ be the prediction rule and stopping rule that $\eas$-learns $(\mathcal{P},\ell)$ with $\eta=\frac{1}{4}$. Consider the Bernoulli 1/2 source. Since $\tau$ stops finitely almost surely on all sources, there exists a number $N$ such that $\tau$ has stopped before step $N$ w.p. $\ge \frac{3}{4}$ when the Bernoulli 1/2 source is in force. 

Let $A$ be the event that $\tau$ has stopped before $N$ and $\Phi$ takes value $1$ at step $N+1$. Observe that $p(A)\ge \frac{1}{2}$ by union bound. Let $p'$ be any source other than the Bernoulli 1/2 source such that $||p'-p||_{TV}\le \frac{1}{8N}$. We have $p'(A)\ge \frac{1}{2}-\frac{1}{8}>\frac{1}{4}$ by Lemma~\ref{nfoldtv}. But $\Phi$ incurs an error at step $N+1$ on any sequence in $A$ whenever $p'$ is in force, which contradicts that $\Phi$ incurs errors after the stopping rule has halted on a set of probability $<1/4$.
\end{example}

\subsection{Unions of learnable classes}
Even finite unions of $\eas$-learnable classes need not be $\eas$-learnable, as the stopping rule for one class need not even stop with probability 1 on sources of the other. It is more interesting to consider nested unions of classes. Clearly finite nested unions of $\eas$-learnable classes are trivially $\eas$-learnable, but countable unions of nested $\eas$-learnable classes may not necessarily be $\eas$-learnable. To see this, consider the subclasses $\mathcal{P}_k$ we constructed in Example~\ref{introexample2}. We know that each of the $\mathcal{P}_k$ is $\eas$-learnable but $\bigcup_{k\ge 1}\mathcal{P}_k$ is not $\eas$-learnable per Example~\ref{predictnot2learn} above.

The following theorem however shows that a countable union of nested $\eas$-learnable classes is always $\eas$-predictable. 
\begin{theorem}
\label{lean2predict}
Let $\mathcal{P}$ be a class of distributions, and $\{\mathcal{P}_i,~i\ge 1\}$ be a nesting of $\mathcal{P}$. If for all $i\ge 1$, $(\mathcal{P}_i,\ell)$ is $\eas$-learnable, then $(\mathcal{P},\ell)$ is $\eas$-predictable.
\end{theorem}
\begin{proof}
The proof is similar to the proof of Lemma~\ref{union2eas}. For any $k,i$, let $\Phi_i^k$ and $\tau_i^k$ be the prediction rule and stopping rule respectively that achieves $\eas$-learnability for $(\mathcal{P}_i,\ell)$ with parameter $\eta=\frac{1}{k^22^i}$. We now construct the $\eas$-prediction rule for $(\mathcal{P},\ell)$ as follows. We partition the prediction into phases, at the beginning, we are at phase $0$ and the prediction is arbitrary. We will go from phase $k-1$ to phase $k$ if at least one of $\tau_i^k$ with $i\ge k$ has stopped. Denote $I_k$ to be the index of the stopping rule in $\{\tau_i^k\}_{i\ge k}$ that stops earliest during phase $k-1$. We will use the prediction rule $\Phi_{I_k}^k$ to make the prediction during phase $k$. Suppose the underlying distribution is $p\in \mathcal{P}_j$ for some $j$. We know that for each phase $k\ge 1$ w.p. $1$ there must be some stopping rule in $\{\tau_i^k\}_{i\ge 1}$ with $i\ge j$ that stops finitely almost surely. Using a union bound, we therefore conclude that all the phases will be finite simultaneously almost surely. Now, by definition of $\eas$-learnability, the probability of making errors at phase $k\ge j$ is upper bounded by $\frac{1}{k^2}$ using a union bound. By the Borel-Cantelli lemma, we will make errors in finitely many phases almost surely, proving that $\mathcal{P}$ is $\eas$-predictable.
\end{proof}
\begin{remark}
Note that the nesting property is required for Theorem~\ref{lean2predict} to hold, that is, it is possible that countable unions of non-nested $\eas$-learnable classes are not $\eas$-predictable. To see this, let $r_1,r_2,\cdots$ be an arbitrary enumeration of rationals in $[0,1]$. We denote $\mathcal{B}_{i,j}=\{x\in [0,1]:x\text{ is irrational and }|x-r_i|\ge \frac{1}{j}\}$, and $\mathcal{P}_{i,j}$ be the class of $\iid$ Bernoulli processes with parameters in $\mathcal{B}_{i,j}$. Denote $\mathcal{P}=\bigcup_{i,j\ge 1}\mathcal{P}_{i,j}$. Let $\ell$ be the rationality testing loss of Cover~\cite{cover1973determining} as we introduced in Example~\ref{eg:cover}. It is easy to see that $(\mathcal{P}_{i,j},\ell)$ is $\eas$-learnable (in fact it is $\eta$-predictable for all $\eta>0$). Moreover, the stopping rule for each $\mathcal{P}_{i,j}$ stops finitely almost surely on all sources in $\mathcal{P}$. However, $(\mathcal{P},\ell)$ is not $\eas$-predictable. This follows from Example~\ref{eg:cover} and the fact that the set of irrational numbers in $[0,1]$ is not of Baire first category~\cite[Chapter 2.1]{rudin2006functional}.
\end{remark}

\subsection{Characterization of $\eas$-learnability}
The characterization of $\eas$-learnability is captured by the notion of \emph{identifiability} as follows.
\begin{definition}
\label{identifiid}
Let $\mathcal{U}$ be a collection of probability measures over $\mathcal{X}^\infty$, $\mathcal{V}\subset\mathcal{U}$. The class $\mathcal{V}$ is said to be identifiable in $\mathcal{U}$ if for any $\eta>0$ there exists a stopping rule $\tau_{\eta}$, such that
\begin{itemize}
\item[1.]$\displaystyle p\left(\lim_{n\rightarrow\infty}\tau_{\eta}(X_1^n)=1\right)=1$ for $p\in \mathcal{V}$;
\item[2.]$\displaystyle p\left(\lim_{n\rightarrow\infty}\tau_{\eta}(X_1^n)=1\right)\le \eta$ for $p\in \mathcal{U}\backslash\mathcal{V}$.
\end{itemize}
\end{definition}

In other words, the rule almost surely stops on sources in $\mathcal{V}$, but does not stop on sources in $\mathcal{U}\backslash \mathcal{V}$ with the prescribed confidence. If we want to distinguish whether a source is in $\mathcal{V}$ (the null hypothesis) or in $\mathcal{U}\backslash\mathcal{V}$ (alternate), we can do so with asymptotically zero type-I error and arbitrarily small type-II error, hence the term identifiability. 

\begin{example}
\label{identexample}
Let $\mathcal{U}$ be the collection of all $\iid$ processes with
marginal distributions over $[0,1]$, and let $\mathcal{V}\subset
\mathcal{U}$ be the set of distributions whose marginal mean is \emph{not}
equal to $t$ for some fixed $t\in [0,1]$.

We show $\mathcal{V}$ is identifiable in $\mathcal{U}$. To see this,
let $\epsilon_n=\frac{1}{n}$. Consider the following stopping
rule. At stage $n$, we obtain a sample of size $\frac{2\log
  (2^{n+1}/\eta)}{\epsilon_n^2}$ and check whether the empirical
mean is within $\epsilon_n$ distance of $t$. If not, we stop, else
we continue to stage $n+1$.

We show that this stopping rule identifies $\mathcal{V}$ in
$\mathcal{U}$ using Definition~\ref{identifiid}. Suppose the
underlying process has marginal mean equal to $t$. By Hoeffding bound,
with probability at most $\eta/2^n$, the empirical mean will be
outside distance $\epsilon_n$ to $t$. Therefore, the stopping rule
stops with probability at most $\eta$ by a union bound. If the
marginal mean does not equal $t$, since $\epsilon_n\rightarrow 0$, the
probability that the empirical mean will be within distance
$\epsilon_n$ to $t$ is at most $\frac{\eta}{2^{n}}$. By
Borel-Cantelli lemma, this happens only finitely many times
since $\sum \frac{\eta}{2^n}<\infty$, and the stopping rule
stops almost surely.
\end{example}

The following theorem shows that for $\iid$ processes over $\mathbb{N}$, identifiability is equivalent to relative openness under total variation distance. See Appendix~\ref{sec:c1} for details on the proof.
\begin{theorem}
\label{open2idtf}
Let $\mathcal{U}$ be a collection of $i.i.d.$ processes over
$\mathbb{N}^{\infty}$, $\mathcal{V} \subset\mathcal{U}$. Then $\mathcal{V}$ is
identifiable in $\mathcal{U}$ iff the marginals of $\mathcal{V}$
are relatively open in the marginals of $\mathcal{U}$ under total variation distance.
\end{theorem}

\begin{remark}
\label{remk4}
Note that Theorem~\ref{open2idtf} cannot be generalized to the $\iid$ processes over $\mathbb{R}^{\infty}$. To see this, let $\mathcal{A}$ be the collection of all uniform distributions supported on a \emph{finite} set of $[0,1]$, and $\mathcal{V}$ contains the single uniform distribution $p$ over $[0,1]$. Define $\mathcal{U}=\mathcal{A}\cup \mathcal{V}$. Clearly, $\mathcal{V}$ is relatively open in $\mathcal{U}$ under total variation distance. We show that $\mathcal{V}$ is not identifiable in $\mathcal{U}$. Suppose otherwise, let $\tau_{\eta}$ be a stopping rule as in Definition~\ref{identifiid}. Let $N$ be a number such that $\tau_{\eta}$ stops on $p$ before step $N$ w.p. $\ge 1-\eta/4$. Consider a mixture $\tilde{q}$ of $\mathcal{A}$ as follows, we fist sample a set $S\subset [0,1]$ of size $CN^2$ from $p$, then we sample from $S$ uniformly, where $C$ is a constant. By birthday paradox, for sufficient large $C$, w.p. $\ge 1-\eta/4$, there will be no repeats on a sample of size $N$ from $\tilde{q}$. Conditioning on the event that there is no repeats, the statistic of a sample of size $N$ from $\tilde{q}$ will be the same as the sample of size $N$ from $p$. By union bound, there must be some distribution in $q\in\mathcal{A}$ such that $\tau_{\eta}$ stops on $q$ before step $N$ w.p. $\ge 1-\frac{\eta}{2}>\eta$ for all $\eta<1/2$. Which contradicts to Definition~\ref{identifiid}.
\end{remark}

We now provide the following characterization of $\eas$-learnability. We leave the details of the proof to Appendix~\ref{sec:c1}.

\begin{theorem}
\label{main2b}
Let $\mathcal{P}$ be a collection of probability measures. Then $(\cP,\ell)$
is $\eas$-learnable if for all $\eta>0$ there exists an $\eta$-nesting  $\{\mathcal{P}_i,i\ge 1\}$, where in addition for
all $i$, $\mathcal{P}_i^{\eta}$ is identifiable in $\mathcal{P}$.
Moreover, the condition is necessary if the measures in $\mathcal{P}$ are $\iid$ over $\mathcal{X}^{\infty}$.
\end{theorem}

The tight characterization for $\iid$ sources in Theorem~\ref{main2b} can not be extended to arbitrary
sources. We provide an example below that shows that the characterization in Theorem~\ref{main2b} for $\iid$ sources will not extend to Markov processes with even two states. 
\begin{example}
\label{counterthm4}
We consider the Markov processes with state space $\{0,1\}$. Let
$\mathcal{P}$ be the class that contains the single state $1$ process
$p_0$, and processes $p_{\epsilon}$ with transition probability
$p_{\epsilon}(1|0)=p_{\epsilon}(0|1)=\epsilon$ for all $\epsilon \in
(0,1)$. We assume the initial state of $p_{\epsilon}$ to be uniformly
sampled from $\{0,1\}$.

We define the loss $\ell(p,X_1^n,Y_n)=0$ if there exists $k\le n$ such that $X_k=1$. Else the loss is $1$. Note that the
loss only depends on the samples $X_1^n$ but is independent of the
prediction $Y_n$. Thus the prediction does not affect the loss.

We now observe that the class is $\eas$-learnable. One simply stops if
the initial state is $1$, else we wait until we see $1$.

We now show that the decomposition of Theorem~\ref{main2b} does not
exist for $(\mathcal{P},\ell)$. Suppose otherwise, we have a
decomposition $\{\mathcal{P}_n\}_{n\ge 1}$ such that each $(\cP_n,\ell)$
is $1/4$-predictable. We know that for all $n\ge 1$ there exists a
number $\epsilon_n$ such that for all $p_{\epsilon}\in \mathcal{P}_n$
we have $\epsilon\ge \epsilon_n$. Otherwise, we will not see state $1$
in the sample before a bounded time step if the initial state is $0$
(which happens w.p. $1/2$), thus violating the $1/4$-predictability of
$\mathcal{P}_n$. We now assume $p_0\in\mathcal{P}_k$ for some $k$. We
show that $\mathcal{P}_k$ is not identifiable in $\mathcal{P}$. Taking
the parameter $\eta=1/4$ in Definition~\ref{identifiid}, we know that
any $\tau$ must stop on the all $1$ sequence at some point $N_0$,
since $p_0\in \mathcal{P}_k$. Now, taking any process $p_{\epsilon}$
that is not in $\cP_k$ with $\epsilon<\epsilon_k$ and small enough so
that $(1-\epsilon)^{N_0}\ge 3/4$, we have $\tau$ stops on
$p_{\epsilon}$ with probability at least $3/8>1/4$, contradicting
identifiability.
\end{example}

Note that the reason why construction of Example~\ref{counterthm4} is
possible is because the number of states of the process is not known
\emph{a-priori} (either 1 or 2). Indeed, with arguments similar to 
the necessity part of Theorem~\ref{main2b} in the $\iid$ case, we can show that if
$\mathcal{P}$ is a collection of irreducible finite state Markov processes
with the same number of states then the necessary condition in
Theorem~\ref{main2b} still holds.  We should also emphasize that the
stopping rule derived from the sufficient condition of
Theorem~\ref{main2b} is not meant to be optimal. For specific
problems, there will often be more natural stopping rules.

\section{Variations}
\label{ch2sec5}
In addition to $\eas$-predictability and $\eas$-learnability that we established in the previous sections, we define the natural extensions of weakly $\eas$-predictability and prediction with finite expected loss in this section, and study their relationships with $\eas$-predictability.

\subsection{Weakly $\eas$-predictable}
\begin{definition}
A collection $(\mathcal{P},\ell)$ is said to be \emph{weakly} $\eas$-predictable if for any $\eta>0$ there exists a prediction rule $\Phi^{\eta}$, such that for all $p\in \mathcal{P}$
$$p\left(\sum_{n=1}^{\infty}\ell(p,X_1^n,\Phi^{\eta}(X_1^{n-1}))<\infty\right)\ge 1-\eta.$$
\end{definition}

It is easy to see that the characterization in Theorem~\ref{main1b} holds for weakly $\eas$-predictability as well.

\begin{theorem}
If a collection $(\mathcal{P},\ell)$ is weakly $\eas$-predictable then for all $\eta>0$ there exists a $\eta$-nesting of $\mathcal{P}$.

Conversely, if there exists a universal nesting of $\mathcal{P}$, then $(\mathcal{P},\ell)$ is weakly $\eas$-predictable.
\end{theorem}
\begin{proof}
The sufficiency follows immediately from Theorem~\ref{main1b}, since $\eas$-predictability implies weakly $\eas$-predictability. To prove the necessary condition, one could replicate the same argument as in the proof of Theorem~\ref{main1b}.
\end{proof}

The following example shows that weakly $\eas$-predictable does not imply $\eas$-predictable.

\begin{example}
For any $m\in\mathbb{N}$ we define a set $S_m\subset \mathbb{N}$ such that $|S_m|=m$, and we require $S_m\cap S_{m'}=\emptyset$ for all $m\not=m'\in \mathbb{N}$. Clearly, such a family of sets exists. Now, for any sequences $\textbf{a}=(a_1,a_2,\cdots)\in \mathbb{N}^{\infty}$ and $\textbf{b}=(b_1,b_2,\cdots)\in \mathbb{N}^{\infty}$ with $b_m\in S_m$ for all $m\in\mathbb{N}$, we define a source $p^{(\textbf{a},\textbf{b})}$ as follows.
\begin{itemize}
    \item[1.] The process $p^{(\textbf{a},\textbf{b})}$ has \emph{no} outcome, i.e. the learner can't observe any sample;
    \item[2.] The loss $\ell(p^{(\textbf{a},\textbf{b})},X_1^n,Y_n)=0$ if and only if there exist some $m\in \mathbb{N}$ such that $n\ge a_m$, $Y_n\in S_m$ and $Y_n\not=b_m$.
\end{itemize}

Let $\mathcal{P}$ be the class of all such sources. We now show that there exist randomized prediction rules that achieve weakly $\eas$-predictability for $(\mathcal{P},\ell)$. For any $\eta>0$, the predictor $\Phi^{\eta}$ is defined as follows. Let $M$ be a number such that $M\ge \frac{1}{\eta}$. We select a number $c\in S_M$ uniform randomly at the beginning of the game, and define $\Phi^{\eta}(X_1^{n-1})=c$ for all $n\ge 1$. Clearly, there exists some time step $N$ such that $N\ge a_M$, and the probability that we make errors after step $N$ is at most $\frac{1}{M}\le \eta$.

We now show that any randomized predictor $\Phi$ can't achieve $\eas$-predictability for $(\mathcal{P},\ell)$. Let $\Omega$ be the probability space of the internal randomness of $\Phi$, i.e., for any $\omega\in \Omega$, $\Phi^{\omega}$ is a deterministic strategy. Let $\mu$ be the corresponding probability measure over $\Omega$.

Let $c_1,c_2,\cdots$ be a sequence of numbers such that the probability of $\Phi$ making predictions in $\bigcup_{i\le m}S_i$ after step $c_m$ is upper bounded by $\frac{\eta}{2^m}$ for some $\eta<1$. We show that such sequences exist. Otherwise, there exists some $m$ such that $\Phi$ makes predictions on some number $s\in S_m$ infinite often with positive probability. However, this will violate the $\eas$-predictability of $\Phi$, since it will make infinite errors for the source that takes $b_m=s$.

Let $A$ be the event that, $\Phi$ makes no prediction on $\bigcup_{i\le m}S_i$ after step $c_m$ for all $m\in \mathbb{N}$. By union bound, we have $\mu(A)\ge 1-\eta$. We now take $a_m = c_m$ for all $m\in \mathbb{N}$. And let $b_m$ to be choosing uniform randomly from $S_m$ and independently for different $m$. Denote $\Gamma$ to be the probability space of the selection processes, and $\nu$ be the corresponding probability measure. For any $\gamma\in \Gamma$ we denote $p^{\gamma}$ to be the source that is selected from the process.

For any $(\omega,\gamma)\in \Omega\times \Gamma$, we define $1\{\omega,\gamma\}$ to be the indicator that $\Phi^{\omega}$ makes infinite errors on $p^{\gamma}$. We claim that for any $\omega\in A$ we have
$$\mathbb{E}_{\gamma\sim \nu}1\{\omega,\gamma\}=1.$$
W.o.l.g., we can assume that $\Phi^{\omega}$ makes predictions only on $S_{m+1}$ during phase $m$ of steps from $a_m+1$ to $a_{m+1}$. Now, we have the probability of making errors at phase $m$ is at least $\frac{1}{m}$. Since the $b_m$ is selected independently, we know that the errors at different phases are independent. The claim follows by the converse Borel-Cantelli lemma.

Now, we have shown that
$$\mathbb{E}_{\omega\sim \mu}[\mathbb{E}_{\gamma\sim\nu}1\{\omega,\gamma\}]\ge 1-\eta.$$
By switching order of expectation, we have
$$\mathbb{E}_{\gamma\sim \nu}\mathbb{E}_{\omega\sim\mu}1\{\omega,\gamma\}\ge 1-\eta.$$
This implies that there exists some $\gamma$ such that the probability of $\Phi$ making infinite errors on $p^{\gamma}$ is at least $1-\eta>0$. This violates the $\eas$-predictability of $\Phi$.
\end{example}

\begin{remark}
The above example shows that if the prediction rule is randomized, then weakly $\eas$-predictability does not imply $\eas$-predictability. However, it is still open whether such a separation exists if the prediction rule is required to be deterministic.
\end{remark}

\subsection{Prediction with finite expected loss}
\begin{definition}
A collection $(\mathcal{P},\ell)$ is said to be predictable with finite expected loss if there exists a prediction rule $\Phi$, such that for all $p\in \mathcal{P}$
$$\mathbb{E}_p\left[\sum_{n=1}^{\infty}\ell(p,X_1^n,\Phi(X_1^{n-1}))\right]<\infty.$$
\end{definition}

By the Borel-Cantelli lemma, it is clear that if a collection is predictable with finite expected loss, then it is also $\eas$-predictable. The following example shows that the converse does not hold in general.

\begin{example}
Let $\mathcal{P}$ be the class with a single source $p$ defined as follows. Let $U$ be the uniform distribution over $[0,1]$, we define
$$\forall n\in \mathbb{N},~X_n=U,$$
i.e. all the $X_n$ have the same value that is sampled from $U$. The loss is defined as follows
$$\ell(p,X_1^{n},Y_n)=1\text{ if and only if }X_n\le \frac{1}{n}.$$
Clearly, $(\mathcal{P},\ell)$ is $\eas$-predictable, since with probability $1$ we have $U\not=0$, meaning that we will make no error after step $\lceil\frac{1}{U}\rceil$. However, the collection is not predictable with finite expected loss. To see this, we define $A_n$ to be the event that an error occurred at step $n$. We have
$$\mathbb{E}_p\left[\sum_{n=1}^{\infty}\ell(p,X_1^n,\Phi(X_1^{n-1}))\right]=\sum_{n=1}^{\infty}\mathbb{E}_p[A_n]=\sum_{n=1}^{\infty}\frac{1}{n}=\infty.$$
\end{example}

Clearly, if a collection $(\mathcal{P},\ell)$ is predictable with finite expected loss, then there exists a nesting $\{\mathcal{P}_i,i\ge 1\}$ of $\mathcal{P}$ such that $\forall i\ge 1$ $(\mathcal{P}_i,\ell)$ is predictable with \emph{bounded} expected loss. To do so, we can simply define $$\mathcal{P}_i=\{p\in \mathcal{P}:p\text{ has expected loss}\le i\text{ under prediction rule }\Phi\},$$
where $\Phi$ is a prediction rule that achieves finite expected loss on $(\mathcal{P},\ell)$. However, we are unaware if the converse is true or not. We have the following open problem.

\begin{problem}[Open Problem]
\label{openp1}
Suppose that there is a nesting $\{\mathcal{P}_i,i\ge 1\}$ for a class $\mathcal{P}$, such that $\forall i\ge 1$ $(\mathcal{P}_i,\ell)$ is predictable with bounded expected loss for the same loss function $\ell$. Is the collection $(\mathcal{P},\ell)$ predictable with finite expected loss?
\end{problem}
\begin{remark}
We say a collection $(\mathcal{P},\ell)$ is predictable with bounded expected \emph{tail} loss if there exists a prediction rule $\Phi$, function $\rho:\mathbb{N}^+\rightarrow \mathbb{R}^+$ and constant $N$, such that for all $p\in \mathcal{P}$ and $n\ge N$, we have
$$\mathbb{E}_p\left[\sum_{k=n}^{\infty}\ell(p,X_1^k,\Phi(X_1^{k-1}))\right]\le \rho(n),$$
and $\rho(n)\rightarrow 0$ as $n\rightarrow\infty$. It is not hard to show that Problem~\ref{openp1} is true if we assume $\forall i\ge 1$ the collection $(\mathcal{P}_i,\ell)$ is predictable with bounded expected tail loss. This is satisfied, e.g., by the setting in Theorem~\ref{pureestimation}.
\end{remark}

Note that Problem~\ref{openp1} holds for supervised setting. Let $\mathcal{P}_1,\mathcal{P}_2,\cdots$ be countably many classes such that each $(\mathcal{P}_i,\ell)$ is predictable with finite expected loss using prediction rule $\Phi_i$. We can construct a prediction rule $\Phi$ using the Randomized Weight Majority Algorithm~\cite[Section 4]{littlestone1994weighted} with experts pool $\{\Phi_1,\Phi_2,\cdots\}$. Suppose the underlying distribution $p\in \mathcal{P}_k$ for some $k$ with expected errors of $E_i<\infty$,~\cite[Section 4]{littlestone1994weighted} shows that the expected errors of $\Phi$ is upper bounded by $E_i+\log k<\infty$.

\section{Applications}
\label{exam}
In this section we will characterize several concrete setups with the application of our general characterizations in previous sections.
\subsection{Hypothesis Testing}
\label{hypotest}
For any two disjoint collections $\mathcal{H}_1,\mathcal{H}_2$ of distributions over $\mathbb{R}^d$, the \emph{hypothesis testing} problem is to decide whether some underlying source $p\in \mathcal{H}_1\cup\mathcal{H}_2$ is from $\mathcal{H}_1$ or $\mathcal{H}_2$, by observing $\iid$ samples from $p$. Therefore, using the notations we developed in Section~\ref{ch2sec2}, the class under consideration will include the $\iid$ processes in $\mathcal{H}=\mathcal{H}_1\cup \mathcal{H}_2$, and for any $p\in \mathcal{H}$ the loss is defined as
$$\ellht(p,X_1^n,Y_n)=1\{p\not\in \mathcal{H}_{Y_n}\}.$$
With slight abuse of notation, for any distribution $p$ over $\mathbb{R}^d$, we will use $p$ to denote the $\iid$ process of $p$ as well, and we denote $p^n$ to be the $n$-fold $i.i.d.$ distribution of $p$.

This problem was considered extensively in~\citep{cover1973determining, dembo1994topological}. Informally, Dembo and Peres~\cite{dembo1994topological} showed that the existence of $\eas$-prediction rule for a hypothesis testing problem is \emph{essentially} captured by the notion of what they call $F_{\sigma}$-separability. 

Let $A,B$ be two sets in some metric space with metric $d$. We say $A,B$ are \emph{$F_{\sigma}$-separable} if $A\subset A'$ and $B\subset B'$, $A'\cap B'=\emptyset$ and both $A'$ and $B'$ can be written as countable unions of closed sets (with topology induced by $d$). One of the main results of~\citep{dembo1994topological} is stated below using our notations in Section~\ref{ch2sec2}.

\begin{theorem}[{\cite[Theorem 2]{dembo1994topological}}]
\label{dembothem2}
Let $\mathcal{H}_1$ and $\mathcal{H}_2$ be disjoint collections of distributions over $\mathbb{R}^d$ that are absolutely continuous w.r.t. Lebesgue measure. Denote $\mathcal{H}=\mathcal{H}_1\cup \mathcal{H}_2$. We have\begin{item}
\item[1.]If $\mathcal{H}_1$ and $\mathcal{H}_2$ are $F_{\sigma}$-separable under the weak convergence topology, then $(\mathcal{H},\ellht)$ is $\eas$-predictable.
\item[2.]If for any $p\in \mathcal{H}$, there exists $r>1$ (possibly depends on $p$) such that
$$\int_{\mathbb{R}^d} f_p^r(\x)d\mu<\infty,$$
where $f_p$ is the density of $p$ and $\mu$ is Lebesgue measure. Then $(\mathcal{H},\ellht)$ is $\eas$-predictable iff $\mathcal{H}_1$ and $\mathcal{H}_2$ are $F_{\sigma}$-separable under the weak convergence topology.
\end{item}
\end{theorem}

Moreover, Dembo-Peres~\cite{dembo1994topological} asked whether the condition $r> 1$ in Theorem~\ref{dembothem2}(2) can be removed. We now give a positive answer to this open problem by showing that the condition of the finite $r$th norm of the density can be relaxed to a much weaker condition as defined in the following.
\begin{definition}
\label{unifrombounded}
Let $\mathcal{H}$ be a collection of distributions over $\mathbb{R}^d$ that are absolutely continuous w.r.t. Lebesgue measure. We say $\mathcal{H}$ is \emph{uniformly bounded} if for any $\epsilon>0$ there exists a number $M_{\epsilon}$ such that
$$\forall p\in\mathcal{H},~p(f_p(\x)\ge M_{\epsilon})\le \epsilon,$$
where $f_p$ is the density of $p$.
\end{definition}

Before showing why such a condition is sufficient to establish Theorem~\ref{dembothem2}(2), we first provide a simple alternative proof to Theorem~\ref{dembothem2}(1) by using our general characterization in Section~\ref{sec:easp}. To begin with, we introduce the following lemma, which relates $F_{\sigma}$-separability with the nesting of sets. The proof is left to Appendix~\ref{sec:c2}.
\begin{lemma}
\label{close2compact}
Let $A,B$ be two sets of some metric space with metric $d$. The $A,B$ are $F_{\sigma}$-separable iff there exist nestings $\{A_i,i\ge 1\}$ and $\{B_i,i\ge 1\}$ for $A,B$ respectively, such that
$$\forall i\ge 1,~ \inf\{d(x,y):x\in A_i,y\in B_i\}>0.$$
\end{lemma}

\begin{remark}
Note that, with similar argument as in the proof of Lemma~\ref{close2compact} (see Appendix~\ref{sec:c2}), we can also prove that two sets $A,B$ are $F_{\sigma}$-separable iff there exist nestings $\{A_i,i\ge 1\}$ and $\{B_i,i\ge 1\}$ for $A,B$ respectively, such that for all $i\ge 1$ there is no limit point of $A_i$ in $B_i$ (and vice versa).
\end{remark}
The following lemma will be useful in our following analysis.

\begin{lemma}
\label{finit2bound}
Let $\{\{A_{i,j}\}_{i\in \mathbb{N}},\{B_{i,j}\}_{i\in\mathbb{N}}\}_{j\in \mathbb{N}}$ be nesting such that $A_{i,j}\subset A_{i',j'}$ and $B_{i,j}\subset B_{i',j'}$ whenever $i'\ge i$ and $j'\ge j$. If $d(A_{i,j},B_{i,j})>0$ for all $i,j\in \mathbb{N}$, where $d$ is some metric. Then $A=\bigcup_{i,j\in \mathbb{N}}A_{i,j}$ and $B=\bigcup_{i,j\in\mathbb{N}}B_{i,j}$ are $F_{\sigma}$-separable.
\end{lemma}
\begin{proof}
We only need to show that for all $i,j$, $B$ and $A$ contain no limit point of $A_{i,j}$ and $B_{i,j}$ respectively. Suppose otherwise that there exists some $y\in B$ such that $d(y,A_{i,j})=0$. We have $y\in B_{i',j'}$ for some $i',j'$. By nesting property, we may assume $i=i'$ and $j=j'$. However, this will imply $d(B_{i,j},A_{i,j})=0$, a contradiction.
\end{proof}

For any two distributions $p_1,p_2$ over $\mathbb{R}^d$, the Kolmogorov-Smirnov distance (abbreviate as KS-distance) is defined as
$$|p_1-p_2|_{KS}=|F_{p_1}(\x)-F_{p_2}(\x)|_{KS}\ed\sup_{\x\in \mathbb{R}^d}|F_{p_1}(\x)-F_{p_2}(\x)|,$$
where $F_{p_i}$ is the CDF of $p_i$ for $i\in \{1,2\}$. The following lemma shows that the empirical distribution uniformly estimates any distribution under KS-distance, and is known as Dvoretzky-Kiefer-Wolfowitz Inequality, see e.g.,~\cite{massart1990tight, kiefer1958deviations, naaman2021tight}.
\begin{lemma}[Dvoretzky-Kiefer-Wolfowitz Inequality]
\label{DKWthm}
Let $X_1,X_2,\cdots,X_n$ be $\iid$ samples of distribution $p$ over $\mathbb{R}^d$, $F_n(\x)$ is the CDF of the empirical distribution. Then there exists a constant $C_d$ depends only on $d$ such that
$$p\left(|F_n(\x)-F_{p}(\x)|_{KS}\ge \epsilon\right)\le C_d\exp(-n\epsilon^2).$$
\end{lemma}
Note that the convergence under KS-distance is consistent with weak convergence under the assumption of absolutely continuous w.r.t. Lebesgue measure by Polya's theorem (see, e.g., \cite[Theorem 9.1.4]{athreya2006measure}).

\begin{proof}[Proof of Theorem~\ref{dembothem2}(1)]
By Lemma~\ref{close2compact}, we have nestings $\{A_i,\ge 1\}$, $\{B_i,i\ge 1\}$ of $\cH_1$, $\cH_2$ respectively, such that
$$\forall i\in\mathbb{N},~\inf\{|p_1-p_2|_{KS}:p_1\in A_i, p_2\in B_i\}\ge\frac{1}{i}.$$
By Theorem~\ref{main1b}, we only need to show that $(A_i\cup B_i,~\ellht)$ is $\eta$-predictable for all $i\ge 1$ and $\eta> 0$. Let $F_n$ be the CDF of the empirical distribution with sample of size $n$. By Lemma~\ref{DKWthm}, we can simultaneously make $|F_n(\x)-F_p(\x)|_{KS}\le 1/4i$ with confidence $\ge 1-\eta$ for all $p\in A_i\cup B_i$ by choosing the sample size $n$ large enough. By triangle inequality of KS-distance, one can classify the distributions in $A_i\cup B_i$ successfully with probability at least $1-\eta$, by predicting the class that is closer to $F_n$ under KS-distance.

The theorem follows by observing that convergence under KS-distance is consistent with weak convergence with the absolutely continuous assumption.
\end{proof}

We now show that if the class $\mathcal{H}_1\cup\mathcal{H}_2$ is uniformly bounded as defined in Definition~\ref{unifrombounded}, then the $F_{\sigma}$-separability will be necessary to achieve $\eas$-predictability under the assumption of Theorem~\ref{dembothem2}.

\begin{theorem}
\label{openn}
Let $\mathcal{H}_1,\mathcal{H}_2$ be collections of distributions that are absolutely continuous w.r.t. Lebesgue measure on $\mathbb{R}^d$, and $\mathcal{H}_1\cup \mathcal{H}_2$ is \emph{uniformly bounded}. Then $(\mathcal{H}_1\cup \mathcal{H}_2,~\ellht)$ is $\eas$-predictable if and only if $\mathcal{H}_1,\mathcal{H}_2$ are $F_{\sigma}$-separable under KS-distance.
\end{theorem}
\begin{proof}
The sufficiency follows directly from Theorem~\ref{dembothem2}(1). We now only prove the necessity.

By Theorem~\ref{main1b}, $\eas$-predictability implies that there exist nesting $\{A_i,i\ge 1\}$, $\{B_i, i\ge 1\}$ of $\mathcal{H}_1$, $\mathcal{H}_2$ respectively, such that for all $i\ge 1$, $(A_i\cup B_i,~\ellht)$ is $\frac{1}{8}$-predictable. By Lemma~\ref{close2compact}, it is sufficient to show that for all $i\ge 1$, there is no limit point of $A_i$ in $B_i$ or vice versa. Suppose otherwise, there exist $p_1,p_2,\cdots \in A_i$ and $p\in B_i$, such that $|p_n-p|_{KS}\rightarrow 0$ as $n\rightarrow\infty$. Let $\Phi$ be an arbitrary predictor that achieves $\frac{1}{8}$-predictability of $A_i\cup B_i$ with sample size $i$. Denote $\Phi^{i}$ as the prediction function at step $i$. By absolute continuity, $p_n$ weakly converges to $p$. Therefore, there exists a compact set $S\subset \mathbb{R}^{i\cdot d}$, such that $p(S)\ge \frac{7}{8}$ and $p_n(S)\ge \frac{7}{8}$ for all $n\in \mathbb{N}$ by Prokhorov's theorem (see~\cite[Theorem 5.2]{van2003probability}), where we have used $p,p_n$ to also denote $i$-fold measures of $p,p_n$ over $\mathbb{R}^{i\cdot d}$.  By uniform boundedness of $\mathcal{H}_1\cup \mathcal{H}_2$, there exists a number $M$, such that 
\begin{equation}
\label{unibound}
\forall p\in \mathcal{H}_1\cup \mathcal{H}_2,~p(\{\x\in \mathbb{R}^{i\cdot d},~f_p(\x)\ge M\})\le \frac{1}{16}.
\end{equation}

 By Lusin's theorem~\cite[Theorem 2.24]{rudin1974real}, there exists a continuous function $g$ over $\mathbb{R}^{i\cdot d}$ and a set $E\subset S$, such that $\sup_{\x\in E}|\Phi^i(\x)-g(\x)|\le \frac{1}{4}$ and $\mu(S\backslash E)\le \frac{1}{16M}$ where $\mu(\cdot)$ is the Lebesgue measure over $\mathbb{R}^{i\cdot d}$. Let $\Omega=\{\x\in \mathbb{R}^{i\cdot d}:g(\x)>\frac{3}{2}\}$, we have $\Omega$ is open and $\{\x\in E:\Phi^i(\x)=2\}\subset \Omega$. By~(\ref{unibound}) and $\mu(S\backslash E)\le \frac{1}{16M}$, we have $p(S\backslash E)\le\frac{1}{8}$ and $p_n(S\backslash E)\le\frac{1}{8}$ for all $n\in \mathbb{N}$. By $\frac{1}{8}$-predictability, we have $p(\Omega)\ge p(\Omega\cap E)\ge \frac{7}{8}-\frac{1}{4}=\frac{5}{8}$, since $p(\bar{E})\le\frac{1}{4}$. By weak convergence, we have $\lim\inf p_n(\Omega)\ge p(\Omega)$, since $\Omega$ is open (see~\cite[Theorem 3.2]{van2003probability}). There exists some $p_n$ such that $p_n(\Omega)\ge \frac{1}{2}$, which implies $p_n(\Omega\cap E)\ge \frac{1}{2}-\frac{1}{4}=\frac{1}{4}>\frac{1}{8}$, contradicting the $\frac{1}{8}$-predictability.
\end{proof}
\begin{remark}
Note that, the assumption of uniform boundedness on $\mathcal{H}_1\cup\mathcal{H}_2$ is necessary for the proof of Theorem~\ref{openn} to work. To see this, let $p_n'$ be the uniform distribution over $\{\frac{1}{n},\frac{2}{n},\cdots,\frac{n-1}{n}\}$, and $p$ be the uniform distribution over $[0,1]$. Clearly, we have that $p_n'$ converges to $p$ in distribution. But $p_n'$ is not continuous. This can be easily resolved by choosing some \emph{continuous} distribution $p_n$ such that $|p_n-p_n'|_{KS}\le\frac{1}{n}$ and there exists a set $S_n$ concentrated on $\{\frac{1}{n},\frac{2}{n},\cdots,\frac{n-1}{n}\}$ so that we have $p_n([0,1]\backslash S_n)\le \frac{\eta}{2^n}$ and $p(S_n)\le \frac{\eta}{2^n}$. Let $A=\{p_1,p_2,\cdots\}$ and $B=\{p\}$, we have $p_n$ converges to $p$ in distribution. However, since
$$\sum_{n=1}^{\infty}p(S_n)\le \eta$$
we have $(A\cup B,\ellht)$ is $\eta$-predictable with only one sample by predicting the underlying distribution in $A$ if the sample is in $\bigcup S_n$ and in $B$ otherwise.
\end{remark}

We prove the following corollary, which establishes Corollary~\ref{demboopen}.
\begin{corollary}
\label{boundexp}
Let $G:\mathbb{R}^+\rightarrow \mathbb{R}^+$ be a monotone increasing function such that $\lim_{x\rightarrow\infty}G(x)\rightarrow \infty$, $\mathcal{H}_1,\mathcal{H}_2$ be distributions over $\mathbb{R}^d$ that are absolutely continuous w.r.t. Lebesgue measure. If for all $p\in \mathcal{H}_1\cup \mathcal{H}_2$, we have $\mathbb{E}_{\x\sim p}[G(f_p(\x))]<\infty$, where $f_p$ is the density of $p$. Then $(\mathcal{H}_1\cup \mathcal{H}_2,~\ellht)$ is $\eas$-predictable iff $\mathcal{H}_1,\mathcal{H}_2$ are $F_{\sigma}$ separable with KS-distance.
\end{corollary}
\begin{proof}
By Theorem~\ref{dembothem2} we only need to prove the necessary condition. By Lemma~\ref{close2compact}, Lemma~\ref{finit2bound} and breaking $\mathcal{H}_1\cup \mathcal{H}_2$ into countably subcollections, one may assume, w.l.o.g., $\forall p\in \mathcal{H}_1\cup \mathcal{H}_2,~\mathbb{E}_{\x\sim p}[G(f_p(\x))]\le M$ for some constant $M$. By Theorem~\ref{openn}, we only need to show that $\mathcal{H}_1\cup \mathcal{H}_2$ is uniformly bounded. For any $p\in \mathcal{H}_1\cup \mathcal{H}_2$, we define random variable $Y_p=G(f_p(\x))$. We have, by Markov inequality, $p(Y_p \ge T)\le \frac{M}{T}$. Note that the upper bound is independent of $p$. By letting $T=\frac{M}{\epsilon}$, one can make the probability upper bounded by $\epsilon$. Since $G$ is monotone increasing and goes to infinity, thus invertible on $\mathbb{R}^+$. We now have $p(f_p(\x)\ge G^{-1}(M/\epsilon))\le \epsilon$ for all $p\in \mathcal{H}_1\cup \mathcal{H}_2$ and $\epsilon>0$.
\end{proof}
\begin{remark}
Note that, Theorem~\ref{dembothem2}(2) follows directly from Corollary~\ref{boundexp} by simply taking $G(x)=x^{r-1}$ for some $r>1$. However, the conditions in Corollary~\ref{boundexp} are much weaker, since $G(x)$ can be an arbitrary monotone function so long as $G(x)\rightarrow\infty$ when $x\rightarrow\infty$. One may, e.g., take $G(x)=\log(x)$, which is asymptotically upper bounded by $x^{r-1}$ for all $r>1$.
\end{remark}

We now have the following conjecture, which asserts that a condition such as the uniform boundedness is necessary for Theorem~\ref{dembothem2}(2) to hold. See also Appendix~\ref{alternat} a \emph{complete} characterization for distributions over $\mathbb{N}$.

\begin{conjecture}
\label{conj1}
There exist collections $\mathcal{H}_1,\mathcal{H}_2$ of distributions over $\mathbb{R}^d$ that are absolutely continuous w.r.t. Lebesgue measure, such that $(\mathcal{H}_1\cup \mathcal{H}_2,~\ellht)$ is $\eas$-predictable but $\mathcal{H}_1,\mathcal{H}_2$ are \emph{not} $F_{\sigma}$-separable under $KS$-distance.
\end{conjecture}

\subsection{The Insurance Problem}
\label{INSUR}
We now consider a problem that is introduced in \citep{JMLR:v16:santhanam15a} and \cite{wu2019isit}. Let $\mathcal{P}$ be a class of \iid processes of distributions over $\mathbb{N}$. The loss function is defined as 
$$\ellip(p,X_1^n,Y_n)=1\{X_n>Y_n\},$$ i.e., we would like predict an upper bound for the next sample $X_n$ by observing $X_1,\cdots,X_{n-1}$. With slight abuse of notations, for any \iid process $p\in \mathcal{P}$ we will use $p$ and $X_p$ to denote the underlying marginal distribution and marginal random variable respectively as well. 

We say a class $\mathcal{P}$ of \iid processes over $\mathbb{N}$ is tight if for any $\epsilon>0$ there exist $N_{\epsilon}$ such that for all $p\in \mathcal{P}$ we have $p(X_p\ge N_{\epsilon})\le \epsilon$. The following theorems were shown in \cite{wu2019isit}.
\begin{theorem}[{\cite[Theorem 1]{wu2019isit}}]
\label{isitpred}
The collection $(\mathcal{P},\ellip)$ is e.a.s.-predictable iff there are countably many tight collections $\{\mathcal{P}_i,i\ge 1\}$. such that
$$\mathcal{P}=\bigcup_{i=1}^{\infty}\mathcal{P}_i.$$
\end{theorem}

\begin{theorem}[{\cite[Corollary 3]{wu2019isit}}]
\label{isitinsur}
The collection $(\mathcal{P},\ellip)$ is e.a.s.-learnable iff
$$\mathcal{P}=\bigcup_{i=1}^{\infty}\mathcal{P}_i,$$where each $\mathcal{P}_i$ is tight \emph{and} relatively open (under total variation distance) in $\mathcal{P}$.
\end{theorem}

We now show that the above theorems can be naturally derived from our general characterizations in the previous sections. The crucial part of the proof is the following technical lemma.
\begin{lemma}
\label{INSURETA}
If a collection $(\mathcal{P},\ellip)$ is $\eta$-predictable with $\eta<\frac{1}{4}$, then there exists nesting $\{\mathcal{P}_i,i\ge 1\}$ of $\mathcal{P}$ such that each $\mathcal{P}_i$ is tight and relatively open in $\mathcal{P}$ under total variation distance.
\end{lemma}
\begin{proof}
We show that if collection $(\mathcal{P},\ellip)$ is $\eta$-predictable with $\eta<1/4$, then for any $p\in \mathcal{P}$ there is an neighborhood $\mathcal{N}_p$ of $p$ in $\mathcal{P}$ such that $\mathcal{N}_p$ is tight. Let $m$ be the samples size that achieves $\eta$-predictability of $\mathcal{P}$ with predictor $\Phi$. Suppose for some $p\in \mathcal{P}$, any neighborhood of $p$ is not tight. Let $N=\min\{n:p^m(\max\{X_1^m\}>n)\le 1-4\eta\}$. Let $\epsilon\le \frac{\eta}{m}$, we have by Lemma~\ref{nfoldtv} that $\sup_{q\in\mathcal{B}(p,\epsilon)}q^m(\max\{X_1^m\}\ge N)\le 1-3\eta$. Since $\mathcal{B}(p,\epsilon)$ is not tight, there exist $\delta>0$, so that for any $M\in\mathbb{N}$ there exist $q\in \mathcal{B}(p,\epsilon)$ such that $q(X\ge M)\ge \delta$. Let 
$$a_1=\Phi(\underbrace{N,\cdots,N}_{m}),$$ and $a_n=\max_{X_1^{n+m}\in [a_{n-1}]^{m+n}}\Phi(X_1^{n+m})$. Let $M_{\delta}$ be a number such that $(1-\delta)^{M_{\delta}}<\eta$, and $q_{\delta}\in \mathcal{B}(p,\epsilon)$ with $q_{\delta}(X\ge a_{M_{\delta}})\ge \delta$. Now, by construction, with probability at least $2\eta$, $\Phi$ makes error at step $m+M_{\delta}$ when $q_{\delta}$ is in force. Contradicting to the $\eta$-predictability of $\mathcal{P}$.

Since the the topology of distributions over $\mathbb{N}$ induced by total variation distance is separable, any open covering of $\mathcal{P}$ has a countable subcover. The lemma follows.
\end{proof}

We now ready to prove Theorem~\ref{isitpred} and Theorem~\ref{isitinsur}.

\begin{proof}[Proof of Theorem~\ref{isitpred}]
We first show that if a collection $\mathcal{P}_i$ is tight then $(\mathcal{P}_i,\ellip)$ is $\eta$-predictable for all $\eta>0$. To see this, by definition of tightness, for any $n\ge 1$ there exist a number $N_n$ such that for all $p\in \mathcal{P}$ we have $p(X_p\ge N_n)\le\frac{\eta}{2^n}$. The prediction rule goes as follows, at step $n$ we predict $N_n$. The $\eta$-predictability of $(\mathcal{P}_i,\ellip)$ follows by a union bound. The sufficiency now follows by Theorem~\ref{main1b}.

To prove the necessity, by Theorem~\ref{main1b}, we know that if $(\mathcal{P},\ellip)$ is $\eas$-predictable then there exist nesting $\{\mathcal{P}_i,i\ge 1\}$ of $\mathcal{P}$ such that $(\mathcal{P}_i,\ellip)$ is $\eta$-predictable for some $\eta<\frac{1}{4}$. By Lemma~\ref{INSURETA}, each of $\mathcal{P}_i$ is a countable union of tight classes, the theorem follows.
\end{proof}

\begin{proof}[Proof of Theorem~\ref{isitinsur}]
By Theorem~\ref{open2idtf}, we know that a class $\mathcal{P}_i$ is identifiable in $\mathcal{P}$ if and only if $\mathcal{P}_i$ is relatively open in $\mathcal{P}$ under total variation distance. The theorem follows by Theorem~\ref{main2b}, Lemma~\ref{INSURETA} and the fact that $(\mathcal{P}_i,\ellip)$ is $\eta$-predictable for all $\eta>0$ if $\mathcal{P}$ is tight.
\end{proof}

\subsection{Online Learning}
\label{online}
Let $\mathcal{H}$ be a set of binary measurable functions from $\mathbb{R}^d\rightarrow\{0,1\}$, $\mu$ is an arbitrary distributions on $\mathbb{R}^d$. We consider the following prediction game with two parties, Nature and the leaner, both know $\mathcal{H}$ and $\mu$. At the beginning Nature chooses some $h\in\mathcal{H}$. At times step $n$, Nature independently samples $X_n\sim\mu$ and provides it to the learner. The learner outputs a guess $Y_n$ of $h(X_n)$ based on his previous observations $(X_1,h(X_1)),\cdots, (X_{n-1},h(X_{n-1}))$ and the new sample $X_n$. Nature reveals $h(X_n)$ after the guess $Y_n$ has been made by the learner. The learner incurs a loss at step $n$ if $Y_n\not=h(X_n)$. 

To put the setup into our general framework, for any $\mu$ and $h\in \mathcal{H}$, one may define the underlying random process to be $Z_1,Z_2,\cdots$ such that $Z_i=(X_{i+1}, h(X_{i}))$ (with $Z_1=(0,h(X_1))$ by convention) where $X_1,X_2,\cdots$ are \iid samples of $\mu$. The loss is defined as
$$\ellol(h,Z_1^n,\Phi(Z_1^{n-1}))=1\{\Phi(Z_1^{n-1})\not= h(X_n)\}.$$
For any function classes $\mathcal{H}$, we will denote $\mathcal{H}$ to be the corresponding random processes as well when the underlying distribution $\mu$ is fixed.

The following result was shown in~\cite{wu2021non}. See Appendix~\ref{sec:c3} for a proof with the application of Theorem~\ref{main1b}.
\begin{theorem}
\label{finiteonline}
Let $\mu$ be an arbitrary distribution over $\mathbb{R}^d$, and $\mathcal{H}$ be a class of measurable functions from $\mathbb{R}^d\rightarrow\{0,1\}$. Then $(\mathcal{H},\ellol)$ is $\eas$-predictable if and only if $\mathcal{H}$ is \emph{effectively countable} w.r.t. $\mu$, meaning that there exist a countable class $\mathcal{H}'$ such that for any $h\in \mathcal{H}$ there exist $h'\in \mathcal{H}'$ we have
$$\mathrm{Pr}_{X\sim \mu}[h(X)\not=h'(X)]=0.$$
\end{theorem}

\subsection{Testing matrix properties}
\label{ch3sec3.3}
We now apply our framework of testing functional properties to matrix properties. To begin with, we first prove the following theorem, which is a direct corollary of Theorem~\ref{dembothm1}(1).
\begin{theorem}
\label{close2eas}
Let $\mathcal{H}$ the set of all $\iid$ processes with marginal
distributions over $[0,1]^d$ for some $d\ge 1$. For all $A\subset
[0,1]^d$, we define loss $\ell_A(p,X_1^n,Y_n)=1\{1\{\mathbb{E}_{X\sim p}[X]\in
A\}\not=Y_n\}$, where the prediction $Y_n$ tries to decide whether
$\mathbb{E}_{X\sim p}[X]\in A$ or not. We have $(\mathcal{H},\ell_A)$ is $\eas$-predictable if
$A$ is closed in $[0,1]^d$.
\end{theorem}
\begin{proof}
This follows directly from Theorem~\ref{dembothm1}(1), since $A,\bar{A}$ are $F_{\sigma}$-separable if $A$ is closed, where $\bar{A}$ is the complement of $A$ in $[0,1]^d$.
\end{proof}

For any function $f:[0,1]^d\rightarrow \{0,1\}$, we will be able to
identify a set $A_f=\{\x\in [0,1]^d:f(\x)=1\}$. Let
$A_1,\cdots,A_n\subset [0,1]^d$ be finitely many sets such that
$(\mathcal{H},\ell_{A_i})$ is $\eas$-predictable for all $1\le i\le
n$. Let $g:[0,1]^d\rightarrow \{0,1\}$ be an arbitrary function,
denote $f(\x)=g(1_{A_1}(\x),\cdots,1_{A_n}(\x))$. It is easy to show
that $(\mathcal{H},\ell_{A_f})$ is also $\eas$-predictable.

We now consider the following problem setup. Let $\X$ be a $d\times d$
random matrix such that each entry $\X(i,j)$ is a Bernoulli random
variable. We denote $\mathbb{E}[\X]$ to be a (deterministic) matrix
that takes expectation entry-wise on $\X$. Let $\X_1,\X_2,\cdots$ be
$\iid$ realization of $\X$, which are binary matrices. We will try to
identify the properties of $\mathbb{E}[\X]$ by observing the samples
$\X_1,\X_2,\cdots$. Clearly, we can associate properties of
$\mathbb{E}[\X]$ with subsets of $[0,1]^{d\times d}$. We will denote
$\mathcal{H}$ to be the class of all $\iid$ Bernoulli random matrices
process. We say a property of $\mathbb{E}[\X]$ is $\eas$-predictable
if $(\mathcal{H},\ell_A)$ is $\eas$-predictable where $A\subset
[0,1]^{d\times d}$ is the subset corresponding to the property.

\begin{theorem}
\label{singulareas}
The singularity of $\mathbb{E}[\X]$ is $\eas$-predictable.
\end{theorem}
\begin{proof}
Note that the determinant is a continuous function w.r.t. the entries of the matrix. Thus subsets in $[0,1]^{d\times d}$ corresponding to the determinant being $0$ are closed. The theorem follows by Theorem~\ref{close2eas}.
\end{proof}

\begin{theorem}
\label{app:rank}
To determine if $\mathbb{E}[\X]$ has rank $k$ is $\eas$-predictable
for all $k$.
\end{theorem}
\begin{proof}
Note that to check whether a matrix has rank $k$, one only needs to check the maximum non-singular square submatrix is of dimension $k$. Thus the property can be expressed as a function of finite singularity tests. By Theorem~\ref{singulareas} we know that the property is still $\eas$-predictable.
\end{proof}
The above theorem is surprising because all of $\X_1,\cdots,\X_n$ will have rank $\ge n-2\log n$ with high probability even when the entries are Bernoulli($\frac{1}{2}$) (i.e. $\mathbb{E}[X]$ has rank $1$). To see this, for each $\X_i$, note that the probability of being full rank for the first $n-2\log n$ rows is
$$\prod_{i=2\log n}^{n}(1-2^{-i})\ge 1-\sum_{i=\log 2n}^{n}2^{-i}=1-O\left(\frac{1}{n^2}\right),$$
and the probability that none of $\X_i$, $1\le i\le n$ have rank $\le n-2\log n$ is $\ge 1- O(1/n)$. Yet, Theorem~\ref{app:rank} shows that it is still possible to infer the rank of $\mathbb{E}[\X]$ from the realizations $\X_1, \X_2, \ldots$. 

To unravel the proof in Theorem~\ref{app:rank}, we also give a more algorithmic way of estimating the rank that is correct eventually almost surely. Let $n$ be the sample size we observed, and $\hat{p}_{i,j}$ be the empirical mean of the $(i,j)$th entry. The estimation of rank is given by the following optimization problem.
\begin{align*}
    &\min_X~\text{rank}(X)\\
    &s.t.~\forall i,j,~|X_{i,j}-\hat{p}_{i,j}|\le\frac{\log n}{\sqrt{n}}.
\end{align*}
We show that the estimation given by the rule above converges to $\text{rank}(\mathbb{E}[\textbf{X}])$ w.p. $1$. Note that, by the selection of the threshold of the constraints at $\log n/\sqrt{n}$ and using the Chernoff bound, the probability that the matrix $\mathbb{E}[\textbf{X}]$ is not in the feasible set of the optimization above is at most $O\left(\frac{1}{n^2}\right)$. By the Borel-Cantelli lemma, $\mathbb{E}[\textbf{X}]$ will be in the feasible set eventually almost surely. Therefore, the rule will never over estimate the rank in the limit. 

Denote $k=\text{rank}(\mathbb{E}[\textbf{X}])$. There exists a full rank $k\times k$ submatrix $\textbf{Y}$ of $\mathbb{E}[\textbf{X}]$. By the continuity of the determinant and since $\text{det}(\textbf{Y})\not=0$, there is a neighborhood $\mathcal{N}$ of $\textbf{Y}$ such that any $\textbf{T}\in \mathcal{N}$ has rank $k$. Since $\log n/\sqrt{n}\rightarrow 0$, there is some constant $N$ depending on the neighborhood $\mathcal{N}$ such that for all $n>N$, the feasible set of the optimization restricted to coordinates of $\textbf{Y}$ will be a subset of $\mathcal{N}$. 
Combining this with the observation that $\mathbb{E}[\textbf{X}]$ can be out of the feasible set only finitely many times, we conclude that the rule will eventually output $k$ almost surely.

\begin{theorem}
To determine if $\mathbb{E}[\X]$ has eigenvalues of multiplicity more than $1$ is $\eas$-predictable.
\end{theorem}
\begin{proof}
For any matrix $A$, consider the characteristic polynomial $p_{A}$ of
$A$. We know that the coefficients of $p_A$ are polynomials of the
entries of $A$. We now only need to check if $\text{GCD}(p_A,p'_A)=1$,
where $p'_A$ is the derivative of $p_A$. Note that this can be done by
checking the resultant of $p_A,p'_A$ is zero. Since resultant is
continuous functions of the coefficients, the theorem follows using
Theorem~\ref{close2eas}.
\end{proof}

While one should expect most properties of matrices to be
$\eas$-predictable, we have the following open problem.
\begin{problem}[Open Problem]
\label{open2}
Is determining whether a matrix is diagonalizable $\eas$-predictable?
\end{problem}

\section{Conclusion}
In this paper, we introduced a general framework for the infinite horizon prediction problems that require to be correct eventually almost surely. We provide general characterization for the existence of the $\eas$-prediction rules. Our characterization has a natural interpretation using regularization that relates $\eas$-predictability with the decomposition of a class into uniformly predictable subclass. While our characterization does not completely capture the $\eas$-predictability in the most general case, it allows us to recover prior results in the literature with simple and elementary proofs, and resolves conjectures posed therein. We also provide tight characterisations for several special cases. There are also several open problems that we have mentioned through out this paper which concerns the exact characterization for certain cases.






\begin{appendix}

\section{More on hypothesis testing}
\label{alternat}
In this section, we present more results on the almost surely hypothesis testing framework (see Section~\ref{hypotest}) using our general machinery.
\subsection{Hypothesis testing with discrete domain}
As we have demonstrated in Section~\ref{hypotest}, the characterization of hypothesis testing for continuous distributions is involved, e.g., we are unaware if $F_{\sigma}$-separability under $KS$-distance is necessarily to achieve $\eas$-predictability even assumed absolute continuity. However, we have the following full characterization for discrete distributions. 
\begin{theorem}
\label{discretetest}
Let $\mathcal{H}_1,\mathcal{H}_2$ be collections of distributions over $\mathbb{N}$, then $(\mathcal{H}_1\cup \mathcal{H}_2,\ellht)$ is $\eas$-predictable if and only if $\mathcal{H}_1,\mathcal{H}_2$ are $F_{\sigma}$-separable under total variation distance.
\end{theorem}
We first prove a general lemma.
\begin{lemma}
\label{eas2tvdist}
Let $\mathcal{H}_1,\mathcal{H}_2$ be classes of distributions over $\mathbb{R}^d$. If $(\mathcal{H}_1\cup\mathcal{H}_2,\ellht)$ is $\eta$-predictable for some $\eta\le \frac{1}{4}$, then
$$\inf\{||p_1-p_2||_{TV}:p_1\in \mathcal{H}_1,~p_2\in \mathcal{H}_2\}>0.$$
\end{lemma}
\begin{proof}
Let $i$ be the sample complexity for $(\mathcal{H}_1\cup\mathcal{H}_2,\ellht)$ to achieve $\eta$-predictability, i.e., there exists a predictor such that the probability of making errors after step $i$ is upper bounded by $\eta$ for all $p\in \mathcal{H}_1\cup\mathcal{H}_2$. Suppose the lemma does not hold. We can find $p_1\in \mathcal{H}_1$ and $p_2\in \mathcal{H}_2$, such that $$||p_1-p_2||_{TV}\le \frac{1}{4i}.$$
We have, by Lemma~\ref{nfoldtv}, $||p_n^i-p^i||_{TV}\le \frac{1}{4}$. However, this implies that any predictor will make an error at step $i+1$ w.p. $\ge \frac{3}{8}>\frac{1}{4}$ by Lemma~\ref{lecam}. This will violate the $\eta$-predictability of $(\mathcal{H}_1\cup\mathcal{H}_2,\ellht)$.
\end{proof}
\begin{proof}[Proof of Theorem~\ref{discretetest}]
We first prove the necessity. By Theorem~\ref{main1b}, there exists a nesting $\{\mathcal{G}_i,i\ge 1\}$ for $\mathcal{H}_1\cup\mathcal{H}_2$ such that $(\mathcal{G}_i,\ellht)$ is $\frac{1}{4}$-predictable for all $i\ge 1$. The necessity is followed now by Lemma~\ref{eas2tvdist}.

To prove the sufficiency, by $F_{\sigma}$-separability, there exist nestings $\{A_i,i\ge 1\}$ and $\{B_i,i\ge 1\}$ for $\mathcal{H}_1$ and $\mathcal{H}_2$ respectively, such that
$$\forall i\ge 1,~\inf\{||p_1-p_2||_{TV}:p_1\in A_i,~p_2\in B_i\}\ge \frac{1}{i}.$$ We now show that $(A_i\cup B_i,\ellht)$ is $\eas$-learnable (see Definition~\ref{easl}). The theorem will now follow by Theorem~\ref{lean2predict}. 

For any $\eta>0$, the stopping rule $\tau_{\eta}$ goes as follows. We partition 
$\tau_{\eta}$ into two phases, at phase one, $\tau_{\eta}$ tries to determine if the missing probability mass of the underlying distribution is less than $\frac{1}{4i}$ with confidence at least $1-\frac{\eta}{2}$. This can be done by requesting additional samples to test if new symbols appear in the new coming samples. At phase two, we will estimate the probability mass only on the symbols we have seen so far at the end of phase one, so that the $L_1$ error of the estimation is upper bounded by $\frac{1}{4i}$ with confidence $\ge 1-\frac{\eta}{2}$. This can be done since the symbols we have seen are finite. Now, the prediction at the end of phase two is as follows, if the total variation distance of the estimated distribution is closer to $A_i$ we predict $p\in \mathcal{H}_1$, otherwise, we predict $p\in \mathcal{H}_2$. We will keep the prediction for all steps after phase two of $\tau_{\eta}$. It is easy to verify that such a rule satisfies the requirements of $\eas$-learnability and $\tau_{\eta}$ stops finitely almost surely on all distributions over $\mathbb{N}$.
\end{proof}

\begin{remark}
Note that, the proof of Theorem~\ref{discretetest} above does not work for distributions over $\mathbb{R}$. Since for distribution classes $A,B$ over $\mathbb{R}$, if
$$\inf\{||p_1-p_2||_{TV}:p_1\in A,~p_2\in B\}>0,$$
we cannot conclude that $(A\cup B,\ellht)$ is $\eas$-learnable. See construction in Remark~\ref{remk4}.

Note that, this does not imply that Theorem~\ref{discretetest} does not hold for distributions over $\mathbb{R}$. We leave it as an open problem to determine whether Theorem~\ref{discretetest} holds for arbitrary distributions over $\mathbb{R}$.
\end{remark}

\subsection{Testing properties of first moment}

For any two sets $A,B\subset \mathbb{R}^d$, we would like to determine whether the underlying distribution has first moment in $A$ or $B$ with the premise that the first moment is in $A\cup B$. The following theorem is due to Dembo and Peres~\cite{dembo1994topological}.

\begin{theorem}[{\cite[Theorem 1]{dembo1994topological}}]
\label{dembothm1}
Let $A,B\subset \mathbb{R}^d$ be disjoint sets. Denote $\mathcal{H}_1^r$ and $\mathcal{H}_2^r$ to be classes of all distributions over $\mathbb{R}^d$ with first moments in $A$ and $B$ respectively that have finite $r$th moment. We have:
\begin{itemize}
\item[1.]If $r>1$, then $(\mathcal{H}_1^r\cup \mathcal{H}_2^r, \ellht)$ is $\eas$-predictable iff $A$ and $B$ are $F_{\sigma}$-separable under $L_2$ distance.
\item[2.]If $r=1$, then $(\mathcal{H}_1^r\cup \mathcal{H}_2^r, \ellht)$ is $\eas$-predictable iff $A$ and $B$ are contained in disjoint \emph{open} sets of $\mathbb{R}^d$.
\end{itemize}
\end{theorem}

We now provide an alternative proof of the theorem above by applying our general characterizations. We need the following lemma.
\begin{lemma}
\label{pmoment}
Let $X_1,\cdots X_n$ be $\iid$ random variables with $\mathbb{E}[|X_1|^r]=M<\infty$ where $r>1$ and $\mathbb{E}[X_1]=u$, $\bar{X}=\frac{X_1+\cdots+X_n}{n}$, then for any $\epsilon>0$ we have
$$\mathrm{Pr}\left[|\bar{X}-u|\ge \epsilon\right]\le\frac{C_{\epsilon,M}}{n^{r-1}},$$where $C_{\epsilon,M}$ is a constant that depends only on $\epsilon$ and $M$.
\end{lemma}
\begin{proof}
W.l.o.g., we assume $u=0$. By Bahr-Essen inequality \citep{von1965inequalities}, we have
$$\mathbb{E}[|\bar{X}|^r]\le 2\frac{\sum_{i=1}^n\mathbb{E}[|X_i|^r]}{n^r}= \frac{2Mn}{n^{r}}=\frac{2M}{n^{r-1}}.$$The lemma follows by a simple application of Markov inequality.
\end{proof}

\begin{proof}[Proof of Theorem~\ref{dembothm1}(1)]
 We first prove the sufficiency. Since $A,B$ are $F_{\sigma}$-separable, we have by Lemma~\ref{close2compact} that there exist nestings $\{A_i,i\ge 1\}$, $\{B_i,i\ge\}$ of $A,B$ respectively, such that
$$\forall i\ge 1, \inf\{||\x-\y||_2: \x\in A_i,\y\in B_i\}\ge \frac{1}{i}.$$
We now define $\mathcal{G}_i=\{p\in \mathcal{H}_1^r\cup \mathcal{H}_2^r: \mathbb{E}[|X_p|^{r}]\le i \text{ and } \mathbb{E}[X_p]\in A_i\cup B_i\}$, where $X_p$ is the random variable governed by distribution $p$. Clearly, we have $\mathcal{G}_i\subset \mathcal{G}_{i+1}$ for all $i\ge 1$. We now show that $\mathcal{H}_1^r\cup \mathcal{H}_2^r=\bigcup_{i\ge 1}\mathcal{G}_i$, thus proving that $\{\mathcal{G}_i,i\ge 1\}$ forms a nesting of $\mathcal{H}_1^r\cup \mathcal{H}_2^r$. To see this, let $p\in \mathcal{H}_1^r\cup \mathcal{H}_2^r$ be an arbitrary distribution with $\mathbb{E}[X_p]\in A$ (respectively $\in B$), we have $\mathbb{E}[X_p]\in A_i$ (respectively $\in B_i$) for some $i$. Since $X_p$ also has finite $r$th moment, we have $\mathbb{E}[|X_p|^r]\le j$ for some $j$. Therefore, we have $p\in \mathcal{G}_{\max\{i,j\}}$ by nesting of $A_i$ (respectively $B_i$).

By Theorem~\ref{main1b}, we only need to show that $(\mathcal{G}_i,\ellht)$ is $\eta$-predictable for all $i\ge 1$ and $\eta>0$. By Lemma~\ref{pmoment}, we know that by letting sample size $n$ large enough one can make 
$$p(|\bar{X}_p-\mathbb{E}[X_p]|\ge \frac{1}{2i})\le \eta$$
for all distributions $p\in \mathcal{G}_i$. To achieve the $\eta$-predictability, we will predict at step $n$ that the first moment is in $A_i$ (respectively $B_i$) if $\bar{X}_p$ is closer to $A_i$ (respectively $B_i$), and \emph{retain} the same prediction for all time step $\ge n$.

We now prove the necessity. Let $\mathcal{G}$ be the class of all Gaussian distributions with mean vectors in $A\cup B$ and covariance matrix $I$ (i.e., the identity matrix). We have $(\mathcal{G},\ellht)$ is $\eas$-predictable, since $\mathcal{G}$ is a subset of $\mathcal{H}_1^r\cup \mathcal{H}_2^r$. Now, by Theorem~\ref{main1b}, there exists a nesting $\{\mathcal{G}_i,i\ge 1\}$ of $\mathcal{G}$ such that for all $i\ge 1$, $(\mathcal{G}_i,\ellht)$ is $\frac{1}{4}$-predictable. Define $A_i=\{\x\in A:\exists p\in \mathcal{G}_i~s.t.~\mathbb{E}[X_p]=\x\}$ (and $B_i$ similarly). Clearly, $\{A_i,i\ge 1\}$ forms a nesting of $A$, and $\{B_i,i\ge 1\}$ forms a nesting of $B$. We show that there is no limit point of $A_i$ (respectively $B_i$) in $B$ (respectively $A$). This will complete the proof by Lemma~\ref{close2compact}.

Suppose we have $\x_1,\x_2,\cdots\in A_i$ and $\x\in B_i$ such that $||\x_n-\x||_2\rightarrow 0$ as $n\rightarrow \infty$. Denote $p_n$ to be the Gaussian distribution with mean $\x_n$ and $p$ be the Gaussian distribution with mean $\x$. We have $||p_n-p||_{TV}\rightarrow 0$ as $n\rightarrow\infty$. By Lemma~\ref{eas2tvdist}, this will violate the $\frac{1}{4}$-predictability of $\mathcal{G}_i$.
\end{proof}

\section{Technical lemmas}
\label{techlem}
In this section we collect some technical lemmas that are used in this paper. The following lemma is due to Le Cam~\cite{lecam1973convergence}.
\begin{lemma}[Le Cam's Two Point Method]
\label{lecam}
Let $p_1,p_2$ be two distributions over the same probability space $\mathcal{X}$. Then for any estimator $\Phi:\mathcal{X}\rightarrow \{1,2\}$ we have
$$\max_{i\in \{1,2\}}\mathrm{Pr}_{X\sim p_i}(\Phi(X)\not=i)\ge \frac{1-||p_1-p_2||_{TV}}{2}.$$
\end{lemma}

The following lemma bounds the total variation of $n$-fold $\iid$ distributions from the marginal distributions, which is "folklore" in the literature.
\begin{lemma}
\label{nfoldtv}
Let $p_1,p_2$ be distributions over the same Polish space with Borel $\sigma$-algebra, and $p_1^n, p_2^n$ be the $n$-fold $\iid$ distributions of $p_1,p_2$ respectively. Then
$$||p^n_1-p^n_2||_{TV}\le 1-(1-||p_1-p_2||_{TV})^n.$$
In particular, we have $||p_1^n-p_2^n||_{TV}\le n||p_1-p_2||_{TV}$.
\end{lemma}
\begin{proof}
Consider a perfect coupling $(X_1,X_2)$ of $(p_1,p_2)$. We have
$$\mathrm{Pr}[X_1\not=X_2]=||p_1-p_2||_{TV}.$$
Let $(X^n_1,X^n_2)$ be the $n$-fold $\iid$ copy of $(X_1,X_2)$. We have
\begin{align*}
    \mathrm{Pr}[X_1^n=X_2^n]&=\mathrm{Pr}[X_1=X_2]^n,~\text{by independence}\\
    &=(1-||p_1-p_2||_{TV})^n.
\end{align*}
Now, we have
\begin{align*}
    ||p_1^n-p_2^n||_{TV}&\le \mathrm{Pr}[X_1^n\not=X_2^n],~\text{by coupling lemma}\\
    &=1-\mathrm{Pr}[X_1^n=X_2^n]\\
    &=1-(1-||p_1-p_2||_{TV})^n.
\end{align*}
The last part follows since $(1-x)^n\ge 1-nx$ for all $x\in [0,1]$ and $n\ge 1$.
\end{proof}

The following example demonstrates that $\eas$-predictability does not imply the existence of universal nesting.
\begin{example}
\label{easnot2universal}
  Let $\mathcal{P}$ be the class of all random processes over $\{0,1\}^{\infty}$ such that for any $p\in \mathcal{P}$ there exists a parameter $b\in \{0,1\}$ and strictly monotonically increasing integer sequence $M_1,M_2,\cdots$ with $M_1=1$ satisfying for all $j\ge 1$, $X_{M_j}=X_{M_j+2}=\cdots=X_{M_{j+1}-1}$, and $X_{M_j}^{M_{j+1}-1}$ is \emph{independent} of all \emph{other} random variables in the process, and
$$p(X_{M_j}=b)=1-\frac{1}{(j+1)^2}.$$Note that a process in $\mathcal{P}$ is purely determined by the parameter $b$ and sequence $\{M_j\}_{j\ge 1}$. We will denote a process to be $p_b$ if the parameter is $b$. Clearly, by Borel-Cantelli lemma $\mathcal{P}$ is $e.a.s.$-predictable under the loss $\ell(p_b,X_1^n,Y_n)=1\{Y_n\not=b\}$ by predicting $X_{n-1}$ at each step $n$. We now show that there is no nesting $\{\mathcal{P}_i,i\ge 1\}$ of $\mathcal{P}$ such that $(\mathcal{P}_i,\ell)$ is $\eta$-predictable for all $\eta>0$ and $i\ge 1$. Our approach is a proof by contradiction. Let $\{\mathcal{P}_i,i\ge 1\}$ be a \emph{universal nesting} of $\mathcal{P}$ w.r.t. loss $\ell$, i.e., $(\mathcal{P}_i,\ell)$ is $\eta$-predictable for all $\eta>0$ and $i\ge 1$. We construct a distribution in $\mathcal{P}$ that is not in $\bigcup_{i\ge 1}\mathcal{P}_i$, a contradiction on our supposition that $\cP=\bigcup_{i\ge 1} \mathcal{P}_i$.

For $j\ge 1$, let $R_j$ be a number such that the class $\mathcal{P}_j$ is
$\eta_j$-predictable with a sample of size $R_j$, where $\eta_j$ will
be specified later. W.l.o.g., we may assume $R_j$ is strictly
increasing on $j$. Let $p_0,p_1$ be two distributions in $\mathcal{P}$
that are associated with the sequence $M_1=1$ and $(M_j=R_{j-1}+1)_{j\ge 2}$
with parameter $b=0$ and $b=1$ respectively, i.e. $p_0,p_1$ share the
same partition of independent blocks but with different parameter for $b$.

Let $||p_0^{R_j}-p_1^{R_j}||_{TV}=1-\epsilon_j$, where $p_b^{R_j}$ with $b\in \{0,1\}$ is
the distribution of $p_b$ on the initial length-$R_j$ binary strings. Observe that
the probability of any length-$R_j$ binary string under either $p_0$
or $p_1$ is purely a function of the number $j$ of blocks which are
all-$0$ (or equivalently all-1), and therefore, so does $\epsilon_j$, the
total variation distance. Specifically, $\epsilon_j$ depends only on $j$, not the sequence $\{R_j\}$, or $\{\eta_j\}$, and in particular we can choose
$\eta_j<\frac{\epsilon_j}{2}$.

By Lemma~\ref{lecam}, we know that any prediction rule will make an error
with probability at least $\frac{\epsilon_j}{2}$ on either $p_0$ or
$p_1$ at time step $R_j+1$ if they both belong to
$\mathcal{P}_j$. Since $\epsilon_j/2 > \eta_j$, we conclude that for all $j\ge 1$ at
least one of $p_0$ or $p_1$ cannot be in $\mathcal{P}_j$.

But $\mathcal{P}_{j}\subset \mathcal{P}_{k}$ for all $k\ge j$. Let $t$
be the smallest number that contains one of $p_0$ or $p_1$.  $\cP_t$
cannot contain both, per the argument above, it follows that $\cP_k$
contains that distribution for all $k\ge t$.  However, the argument
above implies that the distribution missing from $\cP_t$ cannot be in
$\union_{k\ge 1} \cP_k$, contradicting the assumption that
$\{\mathcal{P}_i,i\ge 1\}$ is a nesting of $\mathcal{P}$.
\end{example}

The following lemmas are used in Example~\ref{eg:cover}. Let $r_1,r_2,\cdots$ be an arbitrary enumeration of rational numbers in $[0,1]$.
\begin{lemma}
\label{coverl1}
Let $\mathcal{B}_k$ be the set of all Bernoulli processes
with parameters in $$\mathcal{S}_k=\{r_1,\cdots,r_k\}
\cup
\Paren{
[0,1]
\backslash
\bigcup_{i=1}^{\infty}B(r_i,\frac{1}{k2^i})}$$where $B(r_i,\frac{1}{k2^i})$ is the open balls centered at $r_i$ with
radius $\frac{1}{k2^i}$, $r_1,r_2,\cdots$ is an arbitrary enumeration of rational numbers in $[0,1]$. Then $\mathcal{B}_k$ is $\eta$-predictable with irrationality loss for any $k\ge 1$ and $\eta>0$.
\end{lemma}
\begin{proof}
We show that for any $k\in\mathbb{N}$ and $\eta>0$, there
exists $b_{\eta}$ such that $\mathcal{B}_k$ is $\eta$-predictable with sample size $b_{\eta}$. Let $X_1,\cdots,X_n$ be an \iid sample from some
$p\in\mathcal{B}_k$ with $\mathbb{E}[X_i]=\mu$ and
$\bar{X}=\frac{X_1+\cdots+X_n}{n}$. We have $\mathrm{Var}[X_i]\le
1$. Chebyshev's inequality then shows that
$$p\left(|\bar{X}-\mu|\ge \epsilon\right)\le \frac{1}{n\epsilon^2}.$$
Fix $\epsilon=\frac{1}{k2^{k+1}}$. Let $b_\eta$ be a number large enough so 
that $\frac1{b_{\eta}\epsilon^2} < \eta$. Therefore for $n>b_\eta$,
$p\left(|\bar{X}-\mu|\ge \epsilon\right)$ is less than $\eta$. Thus, we can conclude that
$\mathcal{B}_k$ is $\eta$-predictable by simply predicting the
irrationality of element in $\mathcal{S}_k$ that is closest to
$\bar{X}$ at step $b_{\eta}$, retaining the prediction perpetually thereafter.
\end{proof}

\begin{lemma}
\label{coverl2}
Let $\mathcal{B}$ be a class of Bernoulli processes with parameters in $\mathcal{S}$, if $\mathcal{B}$ is $\eta$-predictable w.r.t. the irrationality loss for some $0<\eta<\frac{1}{2}$, then
\begin{equation}
\label{lm2eq1}
\inf\{|x-r|:x,r\in \mathcal{S}\text{ and }r\in \mathbb{Q},~x\in [0,1]\backslash \mathbb{Q}\}>0.
\end{equation}

\end{lemma}
\begin{proof}
By definition of $\eta$-predictability, there exists a number $N_\eta$
and prediction rule $\Phi_{\eta}$ such that $\Phi_{\eta}$ makes no
errors after step $N_{\eta}$ with probability at least $1-\eta$ for
all $p\in \mathcal{B}$. Suppose, otherwise, that the infimum in
equation~(\ref{lm2eq1}) is $0$. We now select two sources $p_0,p_1$
from $\mathcal{B}$ with parameters $b_0,b_1$ respectively, such that
$b_0$ is rational and $b_1$ is irrational and $|b_0-b_1|<
\frac{1-2\eta}{2N_{\eta}}$.

We now have $||p_0^{N_{\eta}}-p_1^{N_{\eta}}||_{TV}< 1-2\eta$, where
$p_i^{N_{\eta}}$ is the distribution of $p_i$ restricted to the first
$N_{\eta}$ samples---using the fact that $||p^N-q^N||_{TV}\le
N||p-q||_{TV}$ for any distributions $p,q$ with $N$-fold $i.i.d.$
distributions $p^N,q^N$. Now, by Lemma~\ref{lecam}, any prediction
rule (in particular $\Phi_{\eta}$) will make an error at time step
$N_{\eta}+1$ with probability $>\eta$ for either $p_0$ or $p_1$. This
contradicts $\eta$ predictability of $\Phi_{\eta}$ on
$\mathcal{B}$.
\end{proof}

\begin{lemma}
\label{coverl3}
Let $\mathcal{S}_1\subset\mathcal{S}_2\subset\cdots\subset\mathcal{S}_k\subset\cdots\subset [0,1]$ be countably many sets, such that
\begin{equation}
  \label{eq:assumption}
  \forall k,~\inf\{|x-r|:x,r\in \mathcal{S}_k\text{ and }r\in \mathbb{Q},~x\in [0,1]\backslash \mathbb{Q}\}>0.
\end{equation}
If $\bigcup_{k\in \mathbb{N}}\mathcal{S}_k$ contains all rational numbers in $[0,1]$, then the irrational numbers in $\mathcal{S}_k$ are nowhere dense in $[0,1]$ for all $k$.
\end{lemma}
\begin{proof}
Suppose otherwise, the set of irrational numbers $\mathcal{I}_k$ in
$\mathcal{S}_k$ is not nowhere dense. By definition, there exists an
interval $[a,b]\subset \text{col}(\mathcal{I}_k)$, where $\text{col}$
denotes for closure. Since the rational numbers in $[0,1]$ are dense,
there exists some rational number $r\in [a,b]$, and therefore $r\in
\text{col}(\mathcal{I}_k)$. Since $r\in\bigcup_{k\in
  \mathbb{N}^+}\mathcal{S}_k$, there exist some $k'\ge k$ such that
$r\in \mathcal{S}_{k'}$. However, we also have
$\mathcal{S}_{k}\subset\mathcal{S}_{k'}$. Which implies that $r$ is
the limit point of irrational numbers in $\mathcal{S}_{k'}$,
contradicting the assumption~\eqref{eq:assumption} .
\end{proof}

\section{Omitted Proofs}
In this section we collection some of the omitted proofs in the main paper.

\subsection{Proofs in Section 5.2}
\label{sec:c1}
\begin{proof}[Proof of Theorem~\ref{open2idtf}]
  Suppose there is a limit point $p$ of $\mathcal{U}\backslash \mathcal{V}$ in $\mathcal{V}$, we show that $\mathcal{V}$ is not identifiable in $\mathcal{U}$. Let $p_1,p_2,\cdots,p_n,\cdots\in \mathcal{U}\backslash \mathcal{V}$ such that $||p_n-p||_{TV}\rightarrow 0$ as $n\rightarrow\infty$, where we have abused the notation $p,p_n$ to denote the marginal distributions of processes $p,p_n$ as well. For any $\eta\le \frac{1}{4}$, we denote $\tau_{\eta}$ as the stopping rule that identifies $\mathcal{V}$ in $\mathcal{U}$. Let $A_m$ be the event that $\tau_{\eta}$ stops before step $m$. We have $p(A_m)\rightarrow 1$ as $m\rightarrow \infty$. Taking $M$ be a number such that $p(A_M)\ge \frac{3}{4}$. Now, since $||p_n-p||_{TV}\rightarrow 0$, we have some $p_N$ such that $||p_N-p||_{TV}\le\frac{1}{4M}$. Therefore, we have $|p(A_M)-p_N(A_M)|\le \frac{1}{4}$, where we have used the inequality $||p^M-p_N^M||_{TV}\le M||p-p_N||_{TV}$ if $p^M,P_N^M$ are the $M$-fold $i.i.d.$ distributions of $p,p_N$ respectively. However, this would imply that $\tau_{\eta}$ stops on $p_N$ finitely with probability $\ge \frac{3}{4}-\frac{1}{4}\ge\frac{1}{4}$, a contradiction.

Suppose now there are no limit points of $\mathcal{U}\backslash\mathcal{V}$ in $\mathcal{V}$. We know that for any point $p\in \mathcal{V}$, there exists a open ball $\mathcal{B}_p$ of $p$ such that $t_p\overset{\Delta}{=}\inf\{||p-q||_{1}:p\in \mathcal{B}_p\text{ and }q\in\mathcal{U}\backslash\mathcal{V}\}>0$. Since the topological space of distributions over $\mathbb{N}$ is separable, there exist countably many distributions $p_1,p_2,\cdots\in \mathcal{V}$ such that $\mathcal{V}=\bigcup_{i\in \mathbb{N}}\mathcal{B}_{p_i}$. Let $\mathcal{V}_n=\bigcup_{i=1}^n\mathcal{B}_{p_i}$, we show that $\mathcal{V}_n$ can be distinguished with $\mathcal{U}\backslash\mathcal{V}$ with arbitrary high confidence by observing bounded samples. We denote $r_i$ to be the radius of ball $\mathcal{B}_{p_i}$ and $t_i=t_{p_i}$. Denote $n_i=\min\{n:p_i(X\ge n)\le t_i/4\}$. Let $q$ be the underlying distribution, we can estimate the empirical frequency $\hat{q}$ of $q$ on $\{1,2,\cdots,n_i\}$ with error at most $t_i/4$ and arbitrary high confidence using bounded samples. The key observation is that, for any $q\in \mathcal{U}\backslash\mathcal{V}$ we have $||q-p_i||_1\ge r_i+t_i$. Meaning that we can distinguish distributions from $\mathcal{B}_{p_i}$ and $\mathcal{U}\backslash\mathcal{V}$ by checking whether $||\hat{q}-p_i||_1\le r_i+t_i/2$ with arbitrary high confidence using bounded samples. A similar argument as in Example~\ref{identexample} completes the proof.
\end{proof}

\begin{proof}[Proof of Theorem~\ref{main2b}]
  We first show that the stated conditions on a collection $\cP$
  and loss $\ell$ are sufficient to guarantee that $\cP$ is
  $\eas$-learnable. Namely, for all $\eta>0$, we will find an estimator
  $\Phi$ and a stopping rule $\tau$ such that for all $p\in\cP$, the
  probability $\Phi$ incurs non-zero loss after $\tau$ stops is $\le \eta$.

  Now the conditions stated imply that for all $\eta>0$, there is an
  identifiable nesting $\mathcal{P}=\bigcup_{n\in
    \mathbb{N}}\mathcal{P}_{n}$ and a sequence of numbers $\sets{ m_n:
    n\ge 1}$ such that each $\mathcal{P}_{n}$ is
  $\frac{\eta}{2}$-predictable with sample size $m_n$.  Since $\cP_n$
  is identifiable, there is a stopping rule $\sigma_n$ that stops
  after a finite time on $\mathcal{P}\backslash \mathcal{P}_n$ with
  probability at most $\eta/2^{n+1}$ and stops finitely almost surely
  on $\mathcal{P}_n$.
  
   We will assume without loss of
  generality that $\sigma_n$ only stops on sequences of length $\ge m_n$.
  
  The stopping rule $\tau$ for $(\cP,\ell)$ stops if for some $n$,
  $\sigma_n$ has stopped.
  Let $N$ be the index of the stopping rule that stops earliest. The prediction for $(\cP,\ell)$
  is now the $\eta/2-$predictor for $\cP_N$, which we call $\Phi_N$.

  For all $n$, define
    \[
      A_n = \sets{X_1^{\infty}: \sigma_n \text{ stops on }X_1^{n}}.
    \]
  We claim that:
  \begin{enumerate}
  \item The stopping rule $\tau$ stops with probability 1. This is because $p\in \mathcal{P}_k$
  for some $k$, we have $\sigma_k$ stops with probability $1$.
  
    \item The probability that $\tau$ stops but $\Phi$ incurs non-zero
      loss is $\le \eta$. The probability that $\tau$ stops but
      $p\notin \cP_N$ is $\le \sum_{i=1}^{\infty} p(A_i) \le \eta/2$.
      Finally, since $\cP_N$ is $\eta/2-$predictable with sample size
      $m_N$, the probability that $\Phi_N$ predicts incorrectly after
      sample size $m_N$ (which we are guaranteed is the case since we
      assumed $\sigma_N$ only stops on sequences with length$\ge m_N$)
      is $\le \eta/2$. The claim follows by union bound.
    \end{enumerate}

Now, suppose $(\cP,\ell)$ is $\eas$-learnable, and $\mathcal{P}$ are $\iid$ measures on $\mathcal{X}^{\infty}$.
For any $\eta>0$, consider the stopping rule $\tau_{\eta/2}$ and
the estimate $\Phi_{\eta/2}$ that $\eas$-learns $(\cP,\ell)$. Let
\[
  T_n = \sets{ X_1^\infty : \tau_{\eta/2}(X_1^n)=1}
\]
be the set of sequences on which $\tau$ has stopped at or before step
$n$. Now for all
$n$, let
\[
  \cP_n = \sets{ p\in \cP : p(T_n) > 1-\eta/2 }.
\]
Clearly, we have $\cP_n\subset \cP_{n+1}$ and that
$\union_{n\ge 1} \cP_n = \cP$.  For all $n$, and for all $p\in\cP_n$,
$\Phi_{\eta/2}$ incurs non-zero loss on samples of length $n$ with
probability $\le \eta/2$ by construction, namely, $\cP_n$ is
$\eta$-predictable by the union bound.

We will now show that $\cP_n$ are identifiable in $\cP$ to wrap up the theorem.

For any assigned confidence $\delta$, we construct a stopping rule
$\tau_{\delta}$ that stops after a finite time with probability $1$
when the underlying process is from $\mathcal{P}_n$ and with
probability at most $\delta$ on processes from
$\mathcal{P}\backslash\mathcal{P}_n$. To do so, we choose an arbitrary
sequence $\{e_m\}$ with $e_m\rightarrow 0$ as
$m\rightarrow\infty$. The stopping rule is partitioned into phases. At
phase $m$ we estimate $p(T_n)$ with confidence $\ge 1-\frac{\delta}{2^m}$
and error $\le e_m/4$, by considering independent sample blocks
of length $n$. If the estimate is larger than $1-\eta/2+e_m$ we stop,
otherwise we continue to phase $m+1$. Now, if we have
$p(T_n)>1-\eta/2$, then there exist some number $M$ such that
$p(T_n)>1-\eta/2+2e_m$ for all $m\ge M$. Therefore, for all $m\ge M$,
with probability at most $\delta/2^{m}$ we will not stop at phase
$m$. By Borel-Cantelli lemma, with probability $1$ we will stop in a
finite time. A similar argument yields that if $p(T_n)\le 1-\eta/2$,
then we stop with probability at most $\delta$.
\end{proof}

\subsection{Proofs in Section 7.1}
\label{sec:c2}
\begin{proof}[Proof of Lemma~\ref{close2compact}]
By definition of $F_{\sigma}$-separability, we have collections of \emph{closed} sets $\{A_i'\}$ and $\{B_i'\}$ such that $A\subset\bigcup_i A_i'$ and $B\subset\bigcup_i B_i'$. Since finite unions of closed sets are closed, we may assume $\{A_i'\}$, $\{B_i'\}$ to be nested. Define
$$A_i=\{x:x\in A_k'\text{ and }d(x,B_k')\ge 1/i,~k\in\mathbb{N}\}$$and
$$B_i=\{y:y\in B_k'\text{ and }d(y,A_k')\ge 1/i,~k\in\mathbb{N}\}.$$
Since $A_i',B_i'$ are closed, we have $A\subset\bigcup_i A_i$ and $B\subset\bigcup_i B_i$. For any $x\in A_i$ and $y\in B_i$, we show that $d(x,y)\ge 1/i$, thus proving the necessary condition. By definition, we have $x\in A_{k_1}'$, $d(x,B_{k_1}')\ge 1/i$ and $y\in B_{k_2}'$, $d(y,A_{k_2}')\ge 1/i$ for some $k_1,k_2\in \mathbb{N}$. W.l.o.g., we may assume $k_1\le k_2$. Since the collections are nested, we have $x\in A_{k_2}'$. Now, since $d(y,A_{k_2}')\ge 1/i$, we have $d(x,y)\ge 1/i$.

To prove the sufficiency, it is sufficient to show that
$$\bigcup_{i\ge 1}\bar{A}_i\cap \bigcup_{i\ge 1}\bar{B}_i=\emptyset,$$
where $\bar{S}$ denotes the closure of set $S$ under metric $d$. Otherwise, there will be some $x\in\bar{A}_i$ and $x\in \bar{B}_j$ for some $i,j$. w.l.o.g., we can assume $j\ge i$. By nesting property, we now have $x\in \bar{A}_j$ as well. However, this will imply that $d(A_j,B_j)=0$, a contradiction.
\end{proof}

\subsection{Proofs in Section 7.3}
\label{sec:c3}
\begin{proof}[Proof of Theorem~\ref{finiteonline}]
For any $h_1.h_2\in \mathcal{H}$, we define a (pseudo)metric $d(h_1,h_2)=\mathrm{Pr}_{X\sim \mu}[h_1(X)\not=h_2(X)]$. Let $\mathcal{H}'=\{h_1',h_2'\cdots\}$ be an enumeration of $\mathcal{H}'$. We define $\mathcal{H}_i=\{h\in \mathcal{H}:d(h,h_i')=0\}$. Clearly $\mathcal{G}_i=\bigcup_{j=1}^i\mathcal{H}_j$ forms a nesting of $\mathcal{H}$ and $(\mathcal{G}_i,\ellol)$ is $\eta$-predictable for all $\eta>0$, the sufficiency follows by Theorem~1.

By Theorem~\ref{main1b} and Lemma~\ref{decomeq}, we known that if $(\mathcal{H},\ellol)$ is $\eas$-predictable then for all $\eta>0$ we have nesting $\{\mathcal{H}_i, i\ge 1\}$ of $\mathcal{H}$ such that $(\mathcal{H}_{i},\ellol)$ is $\eta$-predictable with sample size $\le i$. Fix some $0<\eta<\frac{1}{4}$ and let $\{\mathcal{H}_i,i\ge 1\}$ be the corresponding decomposition. We show that $\mathcal{H}_{i}$ can be partitioned into countably collections such that any two functions within one collection differ by a measure zero set. This will complete the proof. 

We claim that there exists $\delta>0$ such that
$$\inf\{d(h_1,h_2):h_1,h_2\in \mathcal{H}_i\text{ and }d(h_1,h_2)>0\}\ge\delta.$$
Otherwise, there exist $h_1,h_2\in \mathcal{H}_i$ such that $0<d(h_1,h_2)<\delta_i$, where $\delta_i$ is choosing so that $(1-\delta_i)^i>2\eta$. Let $A\subset(\mathbb{R}^{d})^{\infty}$ be the event that $h_1,h_2$ can't be distinguished within $i$ samples. We have $\mathrm{Pr}[A]> 2\eta$. Fix some predictor $\Phi$, let $p_j=\mathrm{Pr}[\Phi\text{ makes error after step }i\text{ on }h_j]$ for $j\in\{1,2\}$. We show that $\max\{p_1,p_2\}>\eta$, which contradicts to the $\eta$-predictability of $\mathcal{H}_i$, thus establishes the claim. To do so, we use a probabilistic argument, let $\textbf{h}$ to be the random variable uniformly choosing from $\{h_1,h_2\}$. We only need to show
\begin{equation}
\label{eqprobarg}
\mathbb{E}_{\textbf{h}}\mathbb{E}_{\x\sim \mu^{\infty}}[1\{\Phi\text{ makes error after step }i\text{ on }\textbf{h}\}\mid A]\ge \frac{1}{2}
\end{equation}
Let $B\subset(\mathbb{R}^{d})^{\infty}$ be the event that there exist a sample $Y$ after step $i$ that first reveals $h_1(Y)\not=h_2(Y)$. We have $\mathrm{Pr}[B]=1$, since $d(h_1,h_2)>0$. Note that
\begin{equation}
\label{eqprobargin1}
\mathbb{E}_{\x\sim \mu^{\infty}}[\mathbb{E}_{\textbf{h}}1\{\Phi\text{ makes error after step }i\text{ on }\textbf{h}\}\mid A, B]\ge\frac{1}{2},
\end{equation}
since event $C=\{\Phi\text{ makes error at sample }Y\text{ on }\textbf{h}\mid A,B\}$ implies the above event and $\mathbb{E}_{\textbf{h}}[1\{C\}]=\frac{1}{2}$. We now have
\begin{equation}
\label{eqprobargin2}
\mathbb{E}_{\x\sim \mu^{\infty}}[\mathbb{E}_{\textbf{h}}1\{\Phi\text{ makes error after step }i\text{ on }\textbf{h}\}\mid A]\ge\frac{1}{2},
\end{equation}
since $\mathrm{Pr}[B]=1$. Note that (\ref{eqprobargin2}) implies (\ref{eqprobarg}) by exchanging order of expectation. 
 
Now, $\forall h_1, h_2\in \mathcal{H}_{i}$ with $d(h_1,h_2)>0$, we have $d(h_1,h_2)\ge \delta$ for some absolute constant $\delta>0$. There exists a countable dense subset $\mathcal{S}$ of all binary measurable functions over $\mathbb{R}^d$ under our (pseudo)metric $d(\cdot,\cdot)$, since Borel $\sigma$-algebra on $\mathbb{R}^d$ is countably generated thus separable. We can associate each $h\in \mathcal{H}_{i}$ an unique element in $\mathcal{S}$ by triangle inequality. Thus, the elements in $\mathcal{H}_{i}$ can be partitioned into countably many collections such that all elements within one collection have metric zero. This establishes the necessary condition.
\end{proof}

\end{appendix}

\bibliography{eas.bib}
\bibliographystyle{imsart-number}

\end{document}